\documentclass[leqno]{scrartcl}
\usepackage[title,titletoc]{appendix}
\usepackage{amsmath, amsthm,amsfonts, amssymb}
\numberwithin{equation}{section}
\usepackage{mathtools}
\usepackage{thmtools}
\usepackage{stmaryrd}
\usepackage{todonotes}
\usepackage{mathrsfs}
\usepackage{enumitem}
\usepackage{comment}
\usepackage{pgfplots}
\pgfplotsset{compat=1.18}
\usepackage{tikz-cd}
\usepackage[hidelinks]{hyperref}
\usepackage[nameinlink]{cleveref}
\usepackage{lmodern}
\usepackage{microtype}
\usepackage{pgfplotstable}

\usepackage[
  backend=biber,
  style=authoryear,
  sorting=nyt,
  maxcitenames=2,
  maxbibnames=99,
  giveninits=true,
  uniquename=init
]{biblatex}
\DeclareDelimFormat{nameyeardelim}{\addcomma\space}
\newcommand{\from}{\colon}
\newcommand{\RR}{\mathbb{R}}
\newcommand{\CC}{\mathbb{C}}

\newcommand{\SSS}{\mathbb{S}}

\newcommand{\lie}[1]{\mathfrak{#1}}
\newcommand{\ljump}{\llbracket {}}
\newcommand{\rjump}{\rrbracket {}}
\newcommand{\harmonic}{\mathscr H}
\newcommand{\tangent}{\mathrm{tan}}
\newcommand{\normal}{\mathrm{nor}}

\DeclareMathOperator{\Ad}{Ad}

\DeclareMathOperator{\supp}{\supp}

\newcommand{\vol}{ \mathrm{vol}}

\DeclareMathOperator{\im}{im}

\theoremstyle{definition}
\newtheorem{definition}{Definition}[section]

\newlength{\forallsep}
\newcommand{\forallspace}{\hspace{\forallsep}}

\theoremstyle{plain}
\newtheorem{thm}[definition]{Theorem}
\newtheorem{lemma}[definition]{Lemma}
\newtheorem{prop}[definition]{Proposition}
\newtheorem{cor}[definition]{Corollary}
\newtheorem{remark}[definition]{Remark}

\RedeclareSectionCommand[
  afterskip=-.5em,
  font=\normalfont\bfseries
]{subsection}

\title{Finite element discretization of Yang--Mills connections}

\title{Finite element discretization of Yang--Mills connections}

\author{%
  \textsc{\normalsize Geir Bogfjellmo}\\
  {\small\itshape
    Department of Mathematics, Norwegian University of Life Sciences}\\
  {\small\itshape
    P.O. Box 5003, 1432 Ås, Norway}\\[0.4em]
  \textsc{\normalsize Charles Curry}\\
  {\small\itshape
    Department of Mathematical Sciences,
    Norwegian University of Science and Technology}\\
  {\small\itshape
    P.O. Box 191, 2802 Gjøvik, Norway}\\[0.4em]
   \textsc{\normalsize and}\\[0.4em]
  \textsc{\normalsize Andreas Myklebust*}\\%
  {\small\itshape
    Department of Mathematics, Norwegian University of Life Sciences}\\
  {\small\itshape
    P.O. Box 5003, 1432 Ås, Norway}\\
    {\normalfont\small *Corresponding author: andreas.myklebust@nmbu.no}
}

\date{}
\begin{document}\maketitle

\begin{abstract}
    We propose a finite element method for Yang--Mills connections on nontrivial principal bundles with abelian structure group. Local connection forms satisfy internal jump conditions induced by the transition functions of the bundle. We discretize these forms in broken finite element exterior calculus spaces and enforce the jump conditions weakly by Lagrange multipliers. We prove well-posedness of the resulting saddle-point problem and derive an a priori convergence estimate. A numerical experiment for the Hopf fibration illustrates the method and exhibits linear convergence when the sphere is approximated by a piecewise linear mesh.
\end{abstract}

\section{Introduction}
    Connections on principal bundles are fundamental objects in mathematics and physics. In differential geometry, they encode parallel transport, while in physics they appear as gauge fields. For a connection on a (pseudo-)Riemannian manifold, the Yang--Mills functional is a measure of the total curvature of the connection. The critical points of this functional are the Yang--Mills connections. The basic example in physics comes from electromagnetism, where solutions to Maxwell's equations correspond to Yang--Mills connections.
    
    We investigate finite element methods for computing Yang--Mills connections. Such methods have been studied previously \parencite{berchenko2021charge,christiansen2006constraint}. Christiansen and Winther explicitly noted that ``Extension to nontrivial principal bundles seems to be a challenging geometric problem and is under investigation'' \parencite[p.~99]{christiansen2006constraint}. Our contribution addresses this challenge for abelian structure groups.

    A direct discretization poses a dimensionality challenge: a connection on a principal $G$-bundle $\pi\from P\to M$ is a $\lie{g}$-valued 1-form on $P$, so discretizing it there quickly encounters the curse of dimensionality in most interesting cases.
    However, the conditions required for a 1-form to be a connection mean that the connection is fully determined by a collection of locally defined $\lie{g}$-valued 1-forms $\omega_i$, one on each trivializing chart $U_i \subset M$. On overlaps, these 1-forms are related by gauge transformations involving the transition functions between the charts. Any finite element discretization must respect this structure and, in particular, must handle the jump conditions imposed by the transition functions across element boundaries.

    The framework of Finite Element Exterior Calculus (FEEC) \parencite{arnold2006finite,arnold2018finite} provides a natural language for discretizing differential forms and their exterior derivatives.

    The key idea of this paper is to work with ``broken'' or ``discontinuous'' FEEC spaces---defined locally on each cell of a triangulation---and to enforce the gauge-compatibility jump conditions weakly via Lagrange multipliers, in the spirit of hybridized finite element methods \parencite{AriHybrid}. Concretely, given a triangulation $\mathcal{T}$ of $M$ and local sections $\sigma_i\from U_i\to P$, we solve the Yang--Mills equation locally on each cell $T_i\subset U_i$ and enforce the tangential and normal jump conditions prescribed by the transition functions $g_{ij}$ on each facet.

    In the present paper, we restrict the problem to the case where the structure group $G$ is abelian. This class already includes physically important gauge theories, most notably electromagnetism, where $G = \mathrm U(1)$. It does not, however, include many other important theories in physics involving groups such as $\mathrm{SU} (2)$ or $\mathrm{SU}(3)$. This assumption makes the Yang--Mills functional quadratic and simplifies the jump conditions, yielding a linear saddle-point problem that can be analyzed in the standard mixed-method framework \parencite{boffi2013mixed}.

    The analysis proceeds from the geometric jump conditions to an infinite-dimensional saddle-point formulation and then to its finite element discretization. After proving well-posedness, we derive an a priori error estimate and validate the method numerically on the Hopf fibration $\pi\from \SSS^3\to \SSS^2$.

\section{Yang--Mills connections}
    This section introduces some background material. Most of the definitions and terminology follow standard treatments \parencite{hamilton2017mathematical,BleeckerDavid1981Gtav,KN61vol1}. Throughout, $M$ is an $n$-dimensional oriented, compact, connected Riemannian manifold without boundary, with metric $h(\cdot,\cdot)$. We also let $G$ be an $a$-dimensional compact Lie group with Lie algebra $\lie{g}=T_eG$, equipped with an $\Ad$-invariant inner product $\langle\cdot,\cdot\rangle_{\lie g}$.
\subsection{Principal bundles.}
    A principal $G$-bundle, $\pi \colon P\rightarrow M$, is a fiber bundle where $G$ acts fiberwise on $P$ from the right. For any $g\in G$, define $r_g\colon P \to P$ by $r_g(p) = p\cdot g$. A \emph{local section} on a domain $U\subset M$ is a smooth function $\sigma\from U\to P$ such that $\pi\circ \sigma\from U\to M$ is the identity on $U$. 
    
    A nontrivial principal bundle is characterized by the nonexistence of a \emph{global section $\sigma\from M\to P$.}
    
    If we have two overlapping domains $U_\alpha $ and $U_\beta$ in $M$ with corresponding sections $\sigma_\alpha$ and $\sigma_\beta$, there exists a \emph{transition function} $g_{\alpha \beta}\colon U_\alpha \cap U_\beta \to G$ such that
    \begin{equation}\label{eq: transition function}
    \sigma_\beta(x) = \sigma_\alpha(x) \cdot g_{\alpha \beta}(x) \quad \forall x\in U_\alpha\cap U_\beta  .
    \end{equation}
    The transition functions can be defined for any two overlapping domains, and they will satisfy the identities 
    $(g_{\alpha \beta})^{-1}(x) = g_{\beta \alpha }(x)$ and $g_{\alpha\gamma}=g_{\alpha\beta}\cdot g_{\beta\gamma}$.
     \subsection{Lie algebra-valued forms.}
    Let  $C^\infty \Lambda^k(M,\lie g) $ denote the set of smooth $k$-forms on $M$ with values in $\lie g$. If $\{e_1, \dotsc, e_a\}$ is a basis for $ \lie g$, any $k$-form 
    $\tau\in C^\infty \Lambda^k(M,\lie g) $ can be written as $\tau = \sum_j \tau^je_j$ with $\tau^j \in C^\infty \Lambda^k(M)$. 
    We define the \emph{wedge bracket} of two forms $\tau\in C^\infty\Lambda^k(M, \lie{g})$ and $ \varphi \in C^\infty\Lambda^\ell(M, \lie{g})$, to be the $(k+\ell)$-form
    \[
        [ \tau \wedge \varphi](X_1, \dots, X_{k+\ell}) = \sum_{i,j}\tau^i \wedge \varphi ^j(X_1, \dots, X_{k+\ell}) [e_i, e_j].
    \]
    Notice that when $G$ is abelian, $[\tau \wedge \varphi ] = 0 $.
    
    We define the pointwise inner product of two Lie algebra-valued forms $\tau_x, \varphi_x \in \Lambda^k(T_x^\ast M)\otimes\lie{g}$ to be
    \begin{equation}\label{eq: innerprod}
        \langle \tau_x, \varphi_x \rangle_x = \sum_{i,j} \langle \tau^i_x, \varphi^j_x\rangle_h \langle e_i, e_j\rangle_{\lie{g}},
    \end{equation}
    where $\langle \cdot, \cdot \rangle_h$ is the inner product of $k$-forms derived from the Riemannian metric, and $\langle \cdot, \cdot\rangle_{\lie{g}}$ is the $\Ad$-invariant inner product on $\lie{g}$. 
    The \emph{exterior derivative} $d\colon C^\infty\Lambda^k(M, \lie g ) \to C^\infty\Lambda^{k+1}(M, \lie g ) $ is defined by taking the ordinary exterior derivative of each component,
    \[
        d\tau = \sum_j (d\tau^j)e_j. 
    \]
    The  \emph{Hodge star operator} $\star \from  C^\infty \Lambda^k(M,\lie g) \to  C^\infty \Lambda^{n-k}(M,\lie g)$ is defined by taking the regular Hodge star component-wise
    \[
        \star \tau  = \sum_j (\star\tau^j)e_j.
    \]
    The \emph{codifferential} $\delta \colon C^\infty \Lambda^{k+1}(M) \to C^\infty \Lambda^{k}(M)$ is defined by \[\delta\tau = (-1)^{nk +1}\star d \star\tau . \]
\subsection{Connections on a principal bundle.}
  
    A \emph{connection} of the principal $G$-bundle $\pi\from P \to M$ is a 1-form $\omega \from TP\to \lie{g}$ which at every point $p \in P$ satisfies
        \begin{subequations}\label{eq: connection definition}
        \begin{align}
        \Ad_{g^{-1}}\omega(X) = & \mathrlap{\omega((r_{g})_\ast X) \, ,}  && \forall X\in T_pP \\
        \shortintertext{and}
         \omega \left ( \frac{d}{dt}\Big|_{t = 0} p \cdot e^{tA} \right)   =& \mathrlap{A \, ,} && \forall A \in \lie{g} \, . 
        \end{align}
        \end{subequations}
        
A connection defines a canonical bundle decomposition $TP=\mathcal{H}\oplus \mathcal{V}$, where $\mathcal{H}=\ker \omega$ is the \emph{horizontal bundle}, and $\mathcal{V}=\ker \pi_*$ is the \emph{vertical bundle}.
    
    While connections are defined as forms on $P$, the conditions in \eqref{eq: connection definition} imply that they are determined by local forms on $M$. 
    Let $\sigma \from U\to \pi^{-1}(U) \subset P$ be a local section on $U\subset M$. If $\omega$ is a connection on $P$, then the pullback $\sigma^\ast\omega$ is a 1-form on $U$ with values in $\lie{g}$, and any 1-form $\tau \in C^\infty\Lambda^1(U, \lie{g})$ is the pullback of a unique connection on $\pi^{-1}(U)$. 
    If we have multiple sections, say for instance that $\{U_\alpha  ,\sigma_\alpha \}_\alpha$ is a cover of $M$ with corresponding sections $\sigma_\alpha \from U_\alpha  \to P$, and we let $\omega_\alpha \coloneqq\sigma^\ast _\alpha \omega$. Then the local forms $\omega_\alpha$ and $\omega_\beta$ with intersecting domains must satisfy
    \begin{equation}\label{eq: transition forms}
    \omega_\beta - \Ad_ {g_{\alpha \beta }^{-1} }  \omega_\alpha   = g_{\alpha \beta }^\ast \theta \quad \text{ on } U_\alpha  \cap U_\beta \, ,
    \end{equation} 
    where $g_{\alpha \beta}\from U_\alpha \cap U_\beta \to G$ is the transition function \eqref{eq: transition function},
    and $\theta$ is the (left-invariant) Maurer--Cartan form on $G$. 
    \begin{lemma}
    \label{lem: compatibility}
    Equation \eqref{eq: transition forms} is compatible with composition, i.e., if
    \begin{subequations}
    \begin{align}
        \omega_\beta - \Ad_ {g_{\alpha \beta }^{-1} }  \omega_\alpha   &= g_{\alpha \beta }^\ast \theta \label{eq: trf1}\\
        \text{and}& \nonumber\\
        \omega_\gamma - \Ad_ {g_{\beta \gamma}^{-1} }  \omega_\beta   &= g_{\beta \gamma}^\ast \theta \label{eq: trf2} \\
        \text{then}& \nonumber\\
         \omega_\gamma - \Ad_ {g_{\alpha \gamma}^{-1} }  \omega_\alpha   &= g_{\alpha \gamma}^\ast \theta \label{eq: trf3}
    \end{align}
    \end{subequations}
    \end{lemma}
    \begin{proof}
        Insert \eqref{eq: trf1} into \eqref{eq: trf2} and use that $g_{\alpha\gamma}= g_{\alpha\beta}g_{\beta \gamma}$, which implies the identities
        \begin{align*}
            \Ad_{g_{\alpha\gamma}^{-1}}&=\Ad_{g_{\beta\gamma}^{-1}} \Ad_{g_{\alpha\beta}^{-1}} \\
            g_{\alpha\gamma}^\ast \theta&= g_{\beta\gamma}^\ast \theta + \Ad_{g_{\beta\gamma}^{-1}}g_{\alpha\beta}^\ast \theta
        \end{align*}
    \end{proof}
    
    If $G$ is a matrix Lie group, we can write \eqref{eq: transition forms} as
    \[
         \omega _\beta-  g_{\alpha \beta}^{-1} \omega_\alpha g_{\alpha \beta}  = g_{\alpha \beta }^{-1} d g_{\alpha \beta} \quad \text{ on } U_\alpha  \cap U_\beta  \, , 
    \]
    which for an abelian group $G$ simplifies to
      \[
    \omega_\beta- \omega_\alpha  =g_{\alpha \beta}^{-1}dg_{\alpha\beta} \quad \text{ on } U_\alpha  \cap U_\beta \, ,
    \]
    because $\Ad_g{} = I$. 
    Conversely, given a collection $\{U_\alpha,\sigma_\alpha\}_\alpha$, local 1-forms $\{\omega_\alpha\}_\alpha$ satisfying \Cref{eq: transition forms} define a connection \parencite[Proposition~1.4]{KN61vol1}. We use this equivalent description because it allows us to work only on the base manifold $M$.
\subsection{Curvature of a connection.}
    Let $\omega$ be a connection on $P$. 
    
    The \emph{curvature} of $\omega$ is denoted by $F_\omega$ and defined by
    \[
        F_\omega :=  d\omega + \frac 1 2 [\omega \wedge\omega]
    \]
    When $G$ is abelian, the formula simplifies to $F_\omega = d\omega $. If there is no room for confusion, we will drop the subscript and simply write $F$.

    We next describe how $F_\omega$ interacts with local sections.
    Let $\sigma_\alpha\from U_\alpha\to P$ and $\sigma_\beta\from U_\beta\to P$ be two sections of $P$ with nonempty overlap $U_\alpha\cap U_\beta$. Then $\sigma_\alpha ^\ast F = F_\alpha$ and $\sigma_\beta ^\ast F = F_\beta$ are related by
    \begin{equation}\label{eq: Ad-equivariance curvature}
    F_\beta  = \Ad_{g^{-1}_{\alpha\beta}} F_\alpha 
    \end{equation}
    on the intersection $U_\alpha \cap U_\beta$ \parencite[Theorem~5.6.3]{hamilton2017mathematical}. Using the $\Ad{}$-invariant inner product $\langle \cdot, \cdot \rangle_\lie g $, we see from \eqref{eq: Ad-equivariance curvature} that at any point $x\in U_\alpha \cap U_\beta$ we have \[
    \| F_\beta\|_x ^2= \| \Ad _{g_{\alpha\beta}^{-1}} F_\alpha \| _x^2 = \|F_\alpha \|_x ^2\, .
    \]
    We will abuse notation and denote this well-defined function by $\| F_\omega\| ^2 $.
    \subsection{Yang--Mills action functional.}
    The \emph{action} $\mathcal S(\omega)$ of a connection is obtained by integrating the squared magnitude of its curvature over the base manifold
    \begin{equation}\label{eq: action}
        \mathcal S ( \omega) =  \frac 1 2 \int_M \| F_\omega \| ^2 \vol \, =  \frac 1 2 \int_M \langle F_{\omega} , F_\omega \rangle _x \vol \, .
    \end{equation}
    We are interested in finding stationary points of the \emph{Yang--Mills action functional} $\mathcal S$. A connection that is a stationary point of $\mathcal S$ is called a \emph{Yang--Mills connection}, and the problem of finding such a connection is the \emph{Yang--Mills variational problem}. Later we will also consider $\mathcal S(\omega)$ for collections of local forms $\omega$ that need not satisfy \eqref{eq: transition forms}. This is useful because the connection conditions will be imposed separately as constraints.

\subsection{Triangulation of the manifold.}
    We now fix a triangulation $\mathcal{T} = \{ T_i \}_i$ of the manifold $M$. We assume that $M$ and $\mathcal{T}$ are sufficiently well-behaved.
    Specifically, we assume that the triangulation forms a simplicial complex and that each $m$-dimensional simplex $T_i$ is contained in an $m$-dimensional submanifold and in a trivializing subset $U_i$ with a fixed section $\sigma_i \from U_i \to \pi^{-1}(U_i)\subset P$.
    We use $\mathcal T$ to denote the set of $n$-cells and $T$ to denote an arbitrary cell. Similarly, $\mathcal F$ denotes the set of $(n-1)$-dimensional facets, and $f$ denotes an arbitrary facet, i.e., the intersection of two cells. If $T_i$, $T_j$, and $f$ all appear in an equation, then $f = T_i \cap T_j$. We have inclusions $i_f\colon f \to M$ and $i_{\partial T}\colon \partial T \to M$.
    The image of any $n$-dimensional simplex $T$, via a chart, is assumed to be a Lipschitz domain.

    We will be interested in forms defined locally on each cell $T_i\in \mathcal{T}$.
    \[C^\infty \Lambda^k(\mathcal T, \lie{g}):= \bigoplus_{T_i \in \mathcal T} C^\infty \Lambda^k(T_i, \lie{g})\]
    Here we do not assume any continuity across cell boundaries. For elements in $C^\infty \Lambda^k(\mathcal T, \lie{g})$, we write $\varphi=(\varphi_{i})_{i}$, where $\varphi_{i}$ is the component on $T_i$.
 \subsection{Triangulated connection forms.}
    We will now describe a connection $\omega$ on $P$ in terms of the triangulation $\mathcal T$.
    Assume each $T_i$ is contained in an open set $U_i$ with corresponding section $\sigma_i \colon U_i \to P$. A connection can be defined by a collection of local forms
    \[
        \{\omega_j= \sigma_j^\ast \omega \in C^\infty \Lambda^1(U_j, \lie{g})\}_j 
    \]
    satisfying
    \begin{equation}\label{eq: Ujumpcondition}
        \omega_j -\Ad _{g_{ij}^{-1} }\omega_i = g_{ij}^\ast \theta
    \end{equation}
   where $g_{ij}$ is the transition function between overlapping $U_i$ and $U_j$.

   Restricting smooth local connection forms to the cells gives forms $\omega_i \in C^\infty \Lambda^1(T_i, \lie{g})$ that satisfy \eqref{eq: Ujumpcondition} on their common facets. On each interface $T_i \cap T_j$, this condition splits into tangential and normal parts. We take these necessary interface conditions as the definition of a triangulated connection form.
    \begin{definition}[Triangulated connection form]
        Let the following be given
            \begin{itemize}
             \item A triangulation $\mathcal T$ of $M$.
            \item For each $T_i\in \mathcal{T}$, an open neighborhood $U_i \subset M$ and a corresponding section $\sigma_i\colon U_i \to P$.
            \end{itemize}
            Then we say that $ \{\omega_i\} \in C^\infty\Lambda^1(\mathcal T,\lie{g}) $ is a \emph{triangulated connection form}, if on each $ f= T_i \cap T_j$ we have
        \begin{subequations}\label{eq: Tangential and normal}
        \begin{align}\label{Tangential}
            \hspace{1.5cm}i^\ast _f\omega_j -\Ad _{g_{ij}^{-1} } i_f^\ast \omega_i =& i_f^\ast g_{ij}^\ast \theta &  & \text{(tangential part)} \\
            \intertext{and}
            \hspace{1.5cm}i_f^\ast (\star  \omega_j) - \Ad _{g_{ij}^{-1} }i_f^\ast(\star \omega_i) =& i_f^\ast  (\star g_{ij}^\ast \theta) & &\text{(normal part)}         \label{Normal}
        \end{align}          
        \end{subequations}
        where $g_{ij}$ is such that $\sigma_j = \sigma_ig_{ij}$ on each $f$. 
    \end{definition}
    For refinements $\mathcal T_h$ of $\mathcal T$, the sections will be inherited from the base triangulation $\mathcal T$. By this we mean that for $\tilde T_j \in \mathcal T_h$,
    \[
        \tilde T_j \subset {T_i} \in \mathcal T \Longrightarrow  \tilde\sigma_j = \sigma_i \, .
    \]
    Thus, after restricting each coarse-cell form to its child cells, triangulated connection forms on $\mathcal T$ also satisfy \eqref{eq: Tangential and normal} on the refined mesh.
\begin{lemma}
If $\{\omega_i\}_i$ is a triangulated connection form, then the local forms determine a unique continuous, piecewise-smooth $\lie{g}$-valued 1-form $\omega$ on $P$ such that $\sigma_i^\ast \omega=\omega_i$ in the interior of each $T_i$. In each cell interior, $\omega$ is a smooth connection.
\end{lemma}
\begin{proof}
    The construction consists of extending each $\omega_i\in C^\infty\Lambda^1(T_i,\lie{g})$ to a form $\tilde{\omega}_i\in C^0\Lambda^1(U_i,\lie{g})$ by using the forms $\omega_j$ on adjacent cells $T_j$ and \eqref{eq: Ujumpcondition}.
    By construction, the forms $\tilde{\omega}_i$ will satisfy \eqref{eq: Ujumpcondition} and thus define a connection on $P$. 
    We need to verify that the construction is unique at each point in $U_i$.

    On the interior of any $T_j$, we define
    \[
    \tilde{\omega}_i\big\lvert_{\mathring{T_j}}= g_{ji}^\ast\theta + \Ad_{g_{ji}^{-1}}\omega_j
    \]
    If $p$ is on the boundary of two (or more) cells, we can choose between two (or more) $\omega_j$ in the equation above. We need to verify that the resulting $\tilde{\omega}_i(p)$ is independent of the choice of $\omega_j$.
    There are two cases
    \begin{enumerate}[label=(\roman*), ref=\roman*]
    \item $p\in T_i\cap T_j$. In this case \eqref{eq: Tangential and normal} ensures that $\omega_i(v) =  (g_{ji}^\ast\theta)(v) + \Ad_{g_{ji}^{-1}}\omega_j(v)$ for every vector $v\in T_p(M)$.
    \item $p\in U_i\cap T_j \cap T_k$. In this case, we have a choice between $\omega_j$ and $\omega_k$ when defining $\tilde{\omega_i}(p)$.
    However, since $\omega_j, \omega_k$ satisfy \eqref{eq: Tangential and normal}, we have
    \[
    g_{ji}^\ast\theta + \Ad_{g_{ji}^{-1}}\omega_j=  g_{ki}^\ast\theta + \Ad_{g_{ki}^{-1}}\omega_k 
    \]
    on $T_j \cap T_k$ by Lemma \ref{lem: compatibility}.
    Thus, the possible choices for defining $\tilde{\omega}_i(p)$ coincide.
    \end{enumerate}
\end{proof}
In general, the resulting $\omega$ is only $C^0$, rather than $C^\infty$, across facets and is therefore not a smooth connection globally.

\subsection{Orientations with respect to the triangulation.}
    To keep track of differences (``jumps'') between cells, we fix a sign convention for every neighboring pair of cells $T_i, T_j$. We assume that each facet $f\in \mathcal F$ is assigned an orientation and hence a well-defined $(n-1)$-volume form.
    
    Given this orientation, there is a unique skew function $\varepsilon \colon \mathcal T \times \mathcal T\to \{-1, 0, 1\}$ satisfying:
    \begin{itemize}
        \item $\varepsilon_{ji}=0$ if $T_i \cap T_j$ is not $(n-1)$-dimensional.
        \item When $f=T_i\cap T_j $ is a $(n-1)$-dimensional facet, then $\varepsilon_{ji} = \pm 1$. 
        \item For every $\varphi \in C^\infty \Lambda^{n-1}(\mathcal T)$, the identity
    \begin{equation}\label{eq: epsdef}
       \sum _{T_i\in \mathcal T} \int _{\partial T_i }i_{\partial T_i }^\ast \varphi_i=\sum_{\varepsilon_{ji} =1} \int_f \varepsilon_{ji} i^\ast_f\varphi_{j} + \varepsilon_{ij}i^\ast_f \varphi_{i} = \sum_{\varepsilon_{ji} =1} \int_f \varepsilon_{ji}( i^\ast_f\varphi_{j} - i^\ast_f \varphi_{i}) 
    \end{equation}
    holds, where in each summand $f = T_i \cap T_j$ and $i_f^\ast$ is the pullback of the injection $i_f\colon f\to M$.
    \end{itemize}

    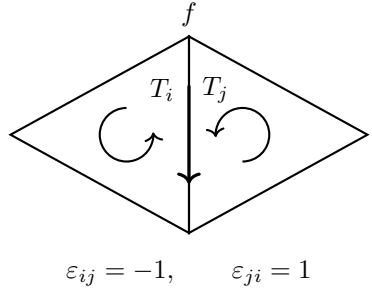
\begin{figure}[!htbp]
	\centering
	\begin{tikzpicture}[scale=1.18, every node/.style={font=\small}]
		\coordinate (A) at (0,1.1);
		\coordinate (B) at (2,0);
		\coordinate (C) at (2,2.2);
		\coordinate (D) at (4,1.1);
		
		\draw[thick] (A) -- (B) -- (C) -- cycle;
		\draw[thick] (D) -- (C) -- (B) -- cycle;
		
		\node at (1.7,1.6) {$T_i$};
		\node at (2.3,1.6) {$T_j$};
		\node at (2,2.45) {$f$};
		
		\draw[thick,->] (1.3,1.4) arc[start angle=90,end angle=365,radius=0.3]; 
		\draw[thick,->] (2.6,0.8) arc[start angle=-90,end angle=185,radius=0.3];  
		
		\draw[very thick,->] (2,1.65) -- (2,0.55);

		\node at (2,-0.45) {$\varepsilon_{ij}=-1,\qquad \varepsilon_{ji}=1$};
	\end{tikzpicture}
	\caption{Two oriented cells sharing a facet $f$.}
    Alt text: Two triangles $T_i$ to the left, and $T_j$ to the right, which share a facet. The facet is marked with an arrow, indicating the orientation of the the facet. Each of the two triangles have a semi-circle arrow, indicating the counter-clockwise orientation. The number $\varepsilon_{ij}$ is -1, because the arrows on $T_i$ and $f$ point in opposite directions. The number $\varepsilon_{ji}$ is 1, because the arrows on $T_j$ and $f$ point in the same direction.
	\label{fig:epsilon-orientation}
\end{figure}

    When we refine $\mathcal T$ to $\mathcal T_h$, the orientation of each facet $f_h\in\mathcal F_h$ is chosen so that, if $f_h\subset f$ for some $f\in\mathcal F$, then $f_h$ inherits the orientation of $f$.

\subsection{Constraint equation with respect to the triangulation.}
    For each $f \in \mathcal F$ the constraint can be written in two equivalent ways.
    \[
        \omega_j - \Ad_{g_{ij}^{-1}}\omega_i = g_{ij}^\ast\theta \Longleftrightarrow
        \omega_i - \Ad_{g_{ji}^{-1}}\omega_j = g_{ji}^\ast\theta
    \]
    We choose to enforce the constraint involving $g_{ij}$ where $\varepsilon_{ji} = 1$. 
    In the following facet integrals, the test forms are $\lie{g}$-valued and the Lie-algebra components of each wedge product are paired using $\langle\cdot,\cdot\rangle_{\lie{g}}$; we suppress this pairing in the notation. With the orientation fixed above, the tangential and normal jumps take values in $\oplus_{f \in \mathcal F} C^\infty\Lambda^1(f,\lie{g})$ and $\oplus_{f \in \mathcal F} C^\infty\Lambda^{n-1}(f,\lie{g})$, respectively. Pairing these jumps with test forms defines two linear maps
    \begin{equation}\label{eq: jumpdef}
    \begin{aligned}
    \ljump \omega^\tangent \rjump_{\Ad}(\varphi) &= \sum _{\varepsilon_{ji} =1}\int_f ( i^\ast_f \omega_j-\Ad_{g_{ij}^{-1}}i^\ast_f\omega_i )\wedge i_f^\ast(\star\varphi)\\  
  \ljump \omega^\normal\rjump_{\Ad}(q) &= \sum _{\varepsilon_{ji} =1} \int_f i^\ast_f q\wedge (i^\ast_f( \star \omega_j)-\Ad_{g_{ij}^{-1}}(i^\ast_f\star\omega_i)) \\
    \end{aligned}
    \end{equation}
    called the \emph{twisted tangential jump} and the \emph{twisted normal jump} of $\omega$, respectively. When $\Ad = I$, we will just call them the \emph{tangential jump} and \emph{normal jump}. In the abelian case, we omit the $\Ad$ subscript and often use the notation $\ljump \omega, \varphi\rjump = \ljump \omega^\tangent \rjump(\varphi)$ and $\ljump q, \omega \rjump = \ljump \omega^\normal\rjump(q)$. The motivation for this notation will become clear in \Cref{sec: Function spaces on the base manifold}.
    
    Similarly, we collect the right-hand sides of \eqref{eq: Tangential and normal} into forms $\psi_1\in\oplus_{f \in \mathcal F} C^\infty\Lambda^1(f,\lie{g})$ and $\psi_2\in\oplus_{f \in \mathcal F} C^\infty\Lambda^{n-1}(f,\lie{g})$:
    \begin{equation}
        \begin{aligned}
        \psi_1 &= \sum _{\varepsilon_{ji} =1} i_f^\ast g_{ij}^\ast \theta \\  
        \psi_2  &= \sum _{\varepsilon_{ji} =1} i_f^\ast  (\star g_{ij}^\ast \theta).
        \end{aligned}
        \label{eq: jumpdefpsi}
    \end{equation}
    with corresponding maps $G_\psi \colon C^\infty\Lambda^{2}(M,\lie{g}) \to \mathbb R $ and $F_\psi \colon C^\infty\Lambda^{0}(M,\lie{g})  \to \mathbb R$ given by
    \begin{equation}\label{eq: prescribed jumps}
    \begin{aligned}
    G_\psi(\varphi) = & 
      \sum_{\varepsilon_{ji} = 1} \int _f  i_f^\ast(g_{ij}^{-1}dg_{ij}) \wedge i_f^\ast(\star \varphi) \\
     F_\psi(q) =& 
      \sum_{\varepsilon_{ji} = 1} \int _f   i_f^\ast q\wedge i_f^\ast(\star g_{ij}^{-1}dg_{ij})
    \end{aligned}
    \end{equation}
    If $\omega \in C^\infty\Lambda^1(\mathcal T,\lie{g})$ and
\[
    (\ljump \omega ^\tangent\rjump_{\Ad} - G_\psi)(\varphi) = \sum_{\varepsilon_{ji} =1} \int_f \big ( (i^\ast_f \omega_j-\Ad_{g_{ij}^{-1}}i^\ast_f\omega_i)- i_f^\ast(g_{ij}^{-1}dg_{ij}) \big ) \wedge i_f^\ast(\star \varphi)  =0 
\]
for all $\varphi \in C^\infty\Lambda^2(M,\lie{g})$, and
\[
    (\ljump \omega ^\normal\rjump_{\Ad} - F_\psi)(q) = \sum_{\varepsilon_{ji} =1} \int_f i_f^\ast(q)\wedge \big ((i^\ast_f( \star \omega_j)-\Ad_{g_{ij}^{-1}}(i^\ast_f\star\omega_i))- i_f^\ast(\star g_{ij}^{-1}dg_{ij}) \big ) =0
\]
for all $q \in C^\infty \Lambda^0(M,\lie{g})$, then \Cref{eq: Tangential and normal} is satisfied. We define two constraint functions $C_\tangent : \omega \mapsto (\ljump \omega ^\tangent\rjump_{\Ad} - G_\psi)$ and $C_\normal : \omega \mapsto (\ljump \omega ^\normal\rjump_{\Ad} - F_\psi)$. For each $\omega$, $C_\tangent(\omega) \colon C^\infty \Lambda ^{2}(M,\lie{g}) \to \mathbb R$ and $C_\normal(\omega) \colon C^\infty \Lambda ^{0}(M,\lie{g}) \to \mathbb R$ are linear maps.

    \todo[inline,disable]{Are these sums really well defined? We are summing over traces taking values in different spaces. The jump form seems well defined to me, and collects the data correctly. But I would prefer to work with the `facet jump' form considered later, where the sum is replace with a direct sum. Moreover it's what we actually work with in the finite element formulation - Charles}
    \todo[inline,disable]{Comment to comment: This wasn't explained that well. It is now changed such that it is clear that we are taking traces to one facet in each summand. Also, I think that if we have $(a_1,0, \dots, 0), (0, \dots, 0,a _n) \in A_1 \oplus \dots \oplus A_n$, then it must be okay to write $a_1 +a_n$ instead of $(a_1, 0 , \dots, 0) + (0, \dots, 0, a_n)$}
    \subsection{Discretized variational problem.}
    The definition of a triangulated connection form allows us to make the action \eqref{eq: action} defined above more concrete.
    Let $\omega \in \bigoplus_i C^\infty\Lambda^1(T_i, \lie{g})$. That is, $\omega$ is a 1-form defined locally on each cell $T_i$. If the constraints \eqref{Tangential} and \eqref{Normal} are satisfied, then $\omega$ is a triangulated connection form.
    The Yang--Mills variational problem becomes
    \begin{equation}\label{eq: YMgeneral}
    \begin{aligned}
    &\text{minimize}\ \mathcal S(\omega)= \frac 1 2 \sum _i\| F_i \|_{L^2\Lambda^2(T_i)}  ^ 2 \\  
    &\text{subject to: \Cref{eq: Tangential and normal} for each facet $f \in \mathcal F$}
    \end{aligned}
    \end{equation}
    The second line in the equation above enforces the triangulated connection conditions and is equivalent to $\omega$ belonging to the constraint manifold defined by $C_\tangent(\omega) = 0$ and $C_\normal(\omega) = 0$. To work with Lagrange multipliers rigorously, we take a completion of $C^\infty\Lambda^1(\mathcal T)$ to form a Hilbert space. We then make sense of $\ljump \omega ^\tangent\rjump$ and $\ljump \omega ^\normal\rjump$ in the Hilbert space setting.
    In the Lagrangian framework, if $\omega$ is a solution to the Yang--Mills variational problem, there exist $p$ and $\lambda$ such that $(\omega,p,\lambda)$ is a stationary point of the Lagrangian
    \begin{equation*}
    \mathcal L(\omega,p,\lambda) \coloneqq \mathcal S(\omega) + \langle C_\tangent(\omega),\lambda \rangle + \langle C_\normal(\omega),p \rangle
    \end{equation*}
    This is the basic idea of our approach. To obtain a computational method, we must specify the function spaces, evaluate the jump terms, and remove the nullspaces of the Lagrangian. The next section constructs the relevant Hilbert-space completions and defines the jump operators on them.
\section{Function spaces on the base manifold}\label{sec: Function spaces on the base manifold}
    This section is primarily concerned with defining traces and jumps. A differential form $\tau$ that is smooth on each $T\in\mathcal T$ has no jumps if and only if $\tau\in H\Lambda^k(M)$; we extend this characterization to forms that lie locally in $H\Lambda^k(T)$.

    Let $M$ be a compact Riemannian manifold, and $\mathcal{T}$ a fixed triangulation of $M$, as defined in the previous section.
    The first challenge when considering a finite element approach to connection forms is the presence of jumps. 
    We will consider the following Hilbert spaces for $k$-forms on $M$. For now, we consider real-valued forms; Lie algebra-valued forms are discussed at the end of the section.
\subsection{Differential form Sobolev spaces.}
For $\tau, \varphi  \in C^\infty _0 \Lambda^k(M)$ their $L^2$-inner product is 
\[
    \langle \tau , \varphi \rangle _{L^2} = \int_M \tau \wedge \star \varphi  = \int_M \langle \tau, \varphi \rangle_ h \vol  \, .
\]
\begin{definition}\label{def: Sobolev spaces}
Define the following Sobolev spaces:
\begin{enumerate}[label=\textup{\roman*.}, ref=\textup{\roman*}]
    \item $L^2\Lambda^k(M) = \overline {C^\infty\Lambda ^k(M)}^{L^2}$, the completion of smooth forms with respect to the $L^2$ induced norm.
    \item $H\Lambda^k(M) \coloneqq \{ \tau \in L^2\Lambda^k(M) : d\tau \in L^2\Lambda ^{k+1}(M) \}$, where $d \tau \in L^2\Lambda^{k+1}(M)$ means that there exists $\alpha \in L^2\Lambda^{k+1}(M)$ such that
\[
    \langle \tau , \delta \varphi \rangle _{L^2}  = \langle \alpha , \varphi \rangle _{L^2} \quad \forall \varphi \in C_0^\infty \Lambda^{k+1}(M).
\]
\item $H^\star \Lambda^k(M)\coloneqq \{\tau \in L^2\Lambda^k(M) : \delta \tau \in L^2\Lambda^{k-1}(M) \}$, where $\delta \tau \in L^2\Lambda^{k-1}(M)$ is again defined weakly.
\item $\displaystyle H\Lambda^k(\mathcal{T}) \coloneqq \oplus_{T_i \in \mathcal T} H\Lambda ^k(T_i)$ where $H\Lambda^k(T_i)$ is defined analogously to $H\Lambda^k(M)$.
\item $H^\star\Lambda^k(\mathcal{T}) \coloneqq \oplus_{T_i \in \mathcal T} H^\star\Lambda ^k(T_i)$
\end{enumerate}
\end{definition}
All the spaces defined above are considered to be contained in $L^2\Lambda^k(M)$. 
We have $H\Lambda^k(M) \subset H \Lambda^k(\mathcal{T})$ and $H^\star \Lambda^k(M) \subset H^\star \Lambda^k(\mathcal{T})$ as closed subspaces.
The spaces $H \Lambda^k(\mathcal{T})$ and $H^\star \Lambda^k(\mathcal{T})$ are sometimes called ``broken'' spaces in literature.
We define two differential operators: 
\begin{align*}
    d^\mathcal T\from H\Lambda^k(\mathcal T) & \to H\Lambda^{k+1}(\mathcal T) &\delta^\mathcal T \from H^\star \Lambda^k(\mathcal T) &\to H^\star\Lambda^{k-1}(\mathcal T) \\
    d^\mathcal T \tau &= \sum_{T_i\in \mathcal T} d\tau_i & \delta ^\mathcal T \tau &= \sum_{T_i\in \mathcal T} \delta \tau_i,
\end{align*}
where $\tau_i$ is the restriction of $\tau$ to $T_i.$
The operators $d^\mathcal{T}$ and $\delta^\mathcal{T}$ are extensions of  $d$ and $\delta$, in the sense that for any $\tau$ in $H\Lambda^k(M)$ or $H^\star\Lambda^{k}(M)$ we have $d^\mathcal T \tau = d \tau $ and $\delta^\mathcal T \tau = \delta \tau$ respectively. The spaces $H\Lambda^k(\mathcal T)$ and $H^\star \Lambda^k(\mathcal T)$ are equipped with the inner products
\begin{align*}
    \langle \tau ,\varphi \rangle _{H} =& \langle \tau,\varphi \rangle _{L^2} + \langle d^\mathcal T \tau , d^\mathcal T\varphi \rangle_{L^2} \\
     \langle \tau ,\varphi \rangle _{H^\star} =& \langle \tau,\varphi \rangle _{L^2} + \langle \delta^\mathcal T \tau , \delta^\mathcal T\varphi \rangle_{L^2}
\end{align*} 
and the usual induced norms obtained by taking square roots. 
\subsection{Hodge decomposition.}
Let $\harmonic \Lambda^k(M)=\left\{\omega \in L^2 \Lambda^k(M) \mid d\omega=0 \text{ and } \delta\omega=0\right\}$ denote the harmonic forms. Throughout, $\harmonic$ is short for ``harmonic'', just as $C^\infty$ is short for ``smooth''. Since the manifold $M$ is compact, the exterior derivative gives rise to a closed Hilbert complex, and we therefore have a Hodge decomposition \parencite[Theorem~4.5]{arnold2018finite}.
\begin{thm}
    Let $M$ be a compact Riemannian manifold without boundary. Then for each $0\leq k \leq n$, we have an $L^2$-orthogonal decomposition
    \[
      L^2\Lambda^k(M) =   dH\Lambda^{k-1}(M) \oplus \delta H^\star \Lambda ^{k+1}(M) \oplus \harmonic \Lambda^k(M)
    \]
\end{thm}
For a general differential form $\tau\in L^2\Lambda^k(M)$ we can write it as the sum $\tau = df + \delta g +  h$. For a general element $\tau$, we let $\harmonic \tau$ denote the projection of $\tau$ onto $\harmonic \Lambda^k(M)$. We then have that $\harmonic \tau = h = \harmonic h $. 

\subsection{Weak jumps and traces.}
\todo[inline,disable]{As per my comment following the definition of the jumps, is this really well-defined?}
Suppose the differential forms appearing in \eqref{eq: jumpdef} are real valued. Then \Cref{eq: jumpdef} is of the form
\begin{equation}\label{eq: weak jump smooth}
    \begin{aligned}
    \ljump \omega^\tangent,\varphi\rjump &= \sum _{\varepsilon_{ji} =1}\int_f ( i^\ast_f \omega_j-i^\ast_f\omega_i )\wedge i_f^\ast(\star\varphi)\\  
  \ljump q, \omega^\normal\rjump &= \sum _{\varepsilon_{ji} =1} \int_f i^\ast_f q\wedge(i^\ast_f( \star \omega_j)-i^\ast_f(\star\omega_i)) \\
    \end{aligned}
\end{equation}
Concentrating on the tangential jump first, we have by \eqref{eq: epsdef}, Stokes theorem, and the identity $\star \delta = (-1)^{k+1}d\star$ for $(k+1)$-forms
\begin{equation}\label{eq: Stokes tangent jump }
    \begin{aligned}
    \ljump \omega^\tangent,\varphi  \rjump=& 
    \sum _{\varepsilon_{ji} =1}\int_f ( i^\ast_f \omega_j-i^\ast_f\omega_i )\wedge i_f^\ast(\star\varphi)\\ =&
    \sum_{T_i \in \mathcal T} \int_{\partial T_i} i^\ast \omega_i \wedge i^\ast (\star\varphi) \\
=& \sum_{T_i \in \mathcal T} \int_{T_i} d\omega_i \wedge \star\varphi  + (-1)^k \omega_i \wedge d(\star \varphi)\\
=& \sum_{T_i\in \mathcal T} \int_{T_i} d\omega_i \wedge \star\varphi  -\omega_i \wedge \star \delta\varphi
    \end{aligned}
\end{equation}
The last line is well-defined even when $\omega \in H\Lambda^k(\mathcal T)$ and $\varphi \in H^\star\Lambda^{k+1}(\mathcal T)$.
We can do a similar thing for the normal jump
\begin{equation}\label{eq: Stokes normal jump }
    \begin{aligned}
    \ljump q,\omega^\normal \rjump=& 
    \sum _{\varepsilon_{ji} =1}\int_f  i_f^\ast q\wedge( i^\ast_f (\star\omega_j)-i^\ast_f(\star \omega_i ))\\ =&
    \sum_{T_i \in \mathcal T} \int_{\partial T_i} i^\ast q \wedge i^\ast (\star \omega_i)  \\
=& \sum_{T_i \in \mathcal T} \int_{T_i} dq \wedge \star\omega_i  + (-1)^{k-1} q \wedge d(\star \omega_i)\\
=& \sum_{T_i\in \mathcal T} \int_{T_i} dq \wedge \star\omega_i  -q \wedge \star \delta\omega_i
    \end{aligned}
\end{equation}
where, again, the last line is well-defined even when $\omega \in H^\star\Lambda^k(\mathcal T)$ and $q\in H\Lambda^{k-1}(\mathcal T)$. Both equations have the same form, which we use to define a bilinear map.
\begin{definition}
     The \emph{jump form} is the bilinear form $\ljump \cdot, \cdot \rjump_{\mathcal T} \colon H\Lambda^k(\mathcal{T})\times H^\star \Lambda^{k+1}(\mathcal T) \to \mathbb R$ defined by 
\begin{equation*}
\begin{aligned}
    \ljump \tau , \varphi \rjump_{\mathcal T} = & \sum_{i} \int_{T_i}d\tau _i\wedge \star \varphi_i  - \tau_i \wedge \star \delta \varphi_i \\ = &
    \langle d^\mathcal T \tau , \varphi \rangle_{L^2} - \langle \tau, \delta ^\mathcal T \varphi \rangle _{L^2} \, . 
\end{aligned}
\end{equation*}

\end{definition}
Here we have taken the domain to be all of $H\Lambda^k(\mathcal{T})\times H^\star \Lambda^{k+1}(\mathcal T)$, even though we are mostly interested in the cases where one of the arguments is in a non-broken space. Since $\mathcal T$ is a fixed triangulation on $M$, we omit the subscript $\mathcal T$ for this triangulation but retain it for refinements $\mathcal T_h$.

\begin{lemma}
The jump form $\ljump \cdot, \cdot \rjump \colon H\Lambda^k(\mathcal{T})\times H^\star \Lambda^{k+1}(\mathcal T) \to \mathbb R$ is continuous. 
\end{lemma}
\begin{proof}
For $\tau \in H\Lambda^k(\mathcal{T})$ and $\varphi \in H^\star \Lambda^{k+1}(\mathcal T)$, we have
\begin{equation*}\begin{aligned}
  | \ljump \tau , \varphi \rjump | \leq  &   \sum _{T_i\in \mathcal T} \big( \|d\tau_i\|_{L^2\Lambda(T_i)} \|\varphi_i\|_{L^2\Lambda(T_i)} + \|\tau_i\|_{L^2\Lambda(T_i)} \|\delta\varphi_i \|_{L^2\Lambda(T_i)} \big )
  \\  \leq &
    \sum _{T_i\in \mathcal T}\big\{ \|\tau_i\|_{L^2\Lambda(T_i)}^2 + \|d\tau_i\|_{L^2\Lambda(T_i)} ^2\big\}^{1/2}\big\{\|\varphi_i\|_{L^2\Lambda(T_i)} ^2+  \|\delta\varphi _i\|_{L^2\Lambda(T_i)}^2\big\}^{1/2}\\ 
  \leq & \Big\{ \sum _{T_i\in \mathcal T}\| \tau_i \|_{H\Lambda(T_i) }^2\Big\}^{1/2}
   \Big\{ \sum _{T_i \in \mathcal T} \| \varphi_i \|^2 _{H^\star\Lambda(T_i)}\Big\} ^{1/2} \\ = &  \| \tau \| _{H\Lambda(\mathcal T)} \| \varphi \|_{H^\star\Lambda(\mathcal T)}.
\end{aligned}
\label{eq: jumpineq}
\end{equation*}
\end{proof}
We can identify the non-broken spaces as the kernels of this bilinear map.
\begin{prop}

\label{lem: nojump}
We can identify
\begin{align*}
    H\Lambda ^k(M) = & \{ \tau \in H\Lambda ^k(\mathcal T) : \ljump \tau, \varphi \rjump = 0 \quad \forall \varphi \in H^\star \Lambda^{k+1}(M)\}\\
    H^\star \Lambda^k(M) = &  \{\tau \in H^\star \Lambda ^ k (\mathcal T) : \ljump q, \tau \rjump = 0 \quad \forall q \in H\Lambda^{k-1}(M) \}
\end{align*}
\end{prop}
\begin{proof}
We only prove the first equality. The proof of the second is analogous.

It is trivial that $H\Lambda^k(M)$ is contained in the kernel.
To prove the other inclusion, we use that $C^\infty \Lambda^{k+1}(M)$ is dense in $H^\star \Lambda^{k+1}(M)$. We therefore have
\[
\begin{aligned}
\ljump \tau , \varphi \rjump &= \sum _{T_i \in \mathcal T} \int_{T_i} \langle d\tau_i, \varphi\rangle -\langle\tau, \delta \varphi _i\rangle \vol\\
&= \int_M \langle \alpha, \varphi\rangle- \langle \tau, \delta \varphi\rangle \vol =0
\end{aligned}
\]
where $\alpha=\sum_{T_i \in \mathcal T}d\tau_i$, for all $\varphi\in C^\infty \Lambda^{k+1}(M)$.
As $\alpha\in L^2\Lambda^{k+1}(M)$, this proves that $\tau \in H\Lambda^{k}(M)$.
\end{proof}
\begin{definition}
    For $\tau \in H\Lambda^{k}(\mathcal{T})$ its \emph{tangential jump}  $\ljump \tau^\tangent\rjump $ is the linear functional acting on $H^\star \Lambda^{k+1}(M)$ by $\varphi \mapsto \ljump \tau, \varphi\rjump$.
Likewise, for $\tau \in H^\star \Lambda^{k}(\mathcal{T})$, the \emph{normal jump} $\ljump \tau ^\normal\rjump \colon \varphi \mapsto  \ljump \varphi, \tau\rjump$ is a linear functional acting on  $H\Lambda^{k-1}(M)$. This generalizes \Cref{eq: jumpdef}.
\end{definition}

\subsection{Refinement of the triangulation.}
We need to consider conforming refinements of a triangulation $\mathcal T$. Let $\mathcal T_h$ denote a refinement of $\mathcal{T}$ where $h$ is the diameter of the triangulation. All the constructions on $\mathcal T$ can also be carried out for $\mathcal T_h$. Let $\ljump  \cdot , \cdot \rjump_{\mathcal T_h}  \colon H\Lambda^k(\mathcal T_h) \times H^\star \Lambda^{k+1}(\mathcal T_h) \to \mathbb R $ denote the refined jump form. It is clear that $H\Lambda ^k(\mathcal T) \subset H\Lambda^k(\mathcal T_h)$, and $H^\star \Lambda^{k+1} (\mathcal T) \subset H^\star\Lambda^{k+1}(\mathcal T_h)$.
The next lemma shows that the refined jump form extends the original jump form in a natural sense.
\begin{lemma}\label{lem: Refining does not add jumps}
    If $\tau\in H\Lambda^k(\mathcal T) $ and $\mu \in H ^\star \Lambda ^{k + 1}(\mathcal T)$, then
    $d^{\mathcal T_h}\tau =  d^{\mathcal T}\tau$ and $\delta^{\mathcal T_h}\mu =  \delta^{\mathcal T}\mu$. Furthermore $\ljump \tau , \mu \rjump_{\mathcal T_h} = \ljump \tau , \mu \rjump_{\mathcal T}$, $\ljump \tau ^\tangent \rjump_{\mathcal T_h} = \ljump \tau^\tangent \rjump_{\mathcal T}$ and $\ljump \mu ^\normal \rjump_{\mathcal T_h} = \ljump \mu^\normal \rjump_{\mathcal T}$
\end{lemma}
\begin{proof}
    The jump identities follow from the differential identities. We prove that $d^{\mathcal T _h}\tau = d^{\mathcal T}\tau$. Let $\varphi \in C^\infty_0\Lambda^{k+1}(T)$, where $T \in \mathcal T_h$. From the definitions of $d^{\mathcal T_h}$ and $d^{\mathcal T}$, we have
       \begin{equation*}
       \langle d^{\mathcal T_h}\tau , \varphi \rangle = \langle \tau, \delta \varphi \rangle =  \langle d^{\mathcal T}\tau , \varphi \rangle 
        \end{equation*}
        which shows that $d^{\mathcal T_h} \tau  = d^{\mathcal T}\tau $. 
\end{proof}
The following refinement analogue of \Cref{lem: nojump} shows that if a form in $H\Lambda^k(\mathcal T_h)$ has the same jumps as a form in $H\Lambda^k(\mathcal T)$, then it belongs to $H\Lambda^k(\mathcal T)$.
\begin{lemma}\label{lem: No refined jump}
Let $\tau^h \in H\Lambda^k(\mathcal T_h)$, and $\tau \in H\Lambda^k(\mathcal T)$. If
\[
    \ljump \tau^h - \tau , \mu\rjump_{\mathcal T_h} = 0  
\]
for all $\mu \in H^\star \Lambda^{k+1}(M)$, then $\tau^h \in H\Lambda^k(\mathcal T)$. 
\end{lemma}
\begin{proof}
    By \Cref{lem: nojump}, $\tau^h - \tau \in H\Lambda^k(M)$, and $\tau^h = \tau +(\tau^h -\tau) \in H\Lambda^k(\mathcal T)$.
\end{proof}

\subsection{Jump spaces and their duals.}
\label{sec: Gspaces}

Having identified the non-broken spaces as kernels of the jump form $\ljump \cdot, \cdot \rjump $, we proceed to define their orthogonal complements as ``jump'' spaces, where tangential and normal jumps live.
\begin{align*}
\mathcal J \Lambda^k(\mathcal T)&= (H\Lambda^k(M))^{\perp_{H}} \subset H\Lambda ^k(\mathcal T), \\
\mathcal J^\star \Lambda^{k+1}(\mathcal{T}) &=  (H^\star \Lambda^{k+1}(M))^{\perp_{H^\star}} \subset H^\star\Lambda ^{k+1}(\mathcal T). \
\end{align*}
We now want to define effective spaces dual to the jump spaces. That is, we define 
$\mathcal K^\star \Lambda^{k+1}(\mathcal T)$ and $\mathcal{K} \Lambda^{k}(\mathcal{T})$ to be spaces such that the restrictions of the jump form
\[\begin{aligned}
\ljump \cdot, \cdot \rjump &\from \mathcal J \Lambda^k(\mathcal T)\times \mathcal K^\star \Lambda^{k+1}(\mathcal T)\to \mathbb{R}\\
\ljump \cdot, \cdot \rjump &\from \mathcal K \Lambda^k(\mathcal T)\times \mathcal J^\star \Lambda^{k+1}(\mathcal T)\to \mathbb{R}
\end{aligned}\]
are non-degenerate. 

\begin{definition}
\label{def: Kdef}
Define $\mathcal K^\star \Lambda^{k+1}(\mathcal T)\subset H^\star \Lambda^{k+1}(M)$ to be the $H^\star$-orthogonal complement to the right kernel of
\[\ljump \cdot, \cdot \rjump\from H\Lambda^k(\mathcal{T})\times H^\star\Lambda^{k+1}(M)\to \mathbb{R}.\]

That is, $\alpha \in \mathcal K^\star \Lambda^{k+1}(\mathcal T) \subset H^\star \Lambda^{k+1}(M)$ if and only if
\[
\ljump \omega, \phi\rjump =0 \ \text{for all }\omega\in H \Lambda^k (\mathcal{T})
\quad\Longrightarrow\quad \langle\alpha, \phi\rangle_{H^\star}=0
\quad\text{for every }\phi\in H^\star\Lambda^{k+1}(M).
\]

Similarly, we define $\mathcal K\Lambda^{k}(\mathcal{T})$ to be the $H$-orthogonal complement of the left kernel of
\[\ljump \cdot, \cdot \rjump\from H\Lambda^k(M)\times H^\star\Lambda^{k+1}(\mathcal T)\to \mathbb{R}.\]
\end{definition}
We can use these spaces to determine if forms have tangential and normal continuity across facets, similar to \Cref{lem: nojump}
\begin{prop}\label{lem: nojump K}
We can identify
\begin{align*}
    H\Lambda ^k(M) = & \{ \tau \in H\Lambda ^k(\mathcal T) : \ljump \tau, \varphi \rjump = 0 \quad \forall \varphi \in \mathcal K^\star \Lambda^{k+1}(\mathcal T)\}\\
    H^\star \Lambda^k(M) = &  \{\tau \in H^\star \Lambda ^ k (\mathcal T) : \ljump q, \tau \rjump = 0 \quad \forall q \in \mathcal K\Lambda^{k-1}(\mathcal T) \}
\end{align*}
\end{prop}

The spaces $\mathcal K^\star\Lambda^{k+1}(\mathcal T)$ and $\mathcal K\Lambda^k(\mathcal T)$ defined above can also be shown to be solutions to certain partial differential equations.
\begin{thm}
The identifications
 \begin{align*}
        \mathcal K  \Lambda ^k(\mathcal T)  =& \{ \tau \in H\Lambda ^k(M) : d\tau\in H^\star \Lambda ^{k+1}(\mathcal T) \; \text{ and }  \; \delta^\mathcal T d \tau = -\tau \} \\
          \mathcal K^\star  \Lambda ^{k+1}(\mathcal T)  =& \{ \tau \in H^\star \Lambda ^{k+1}(M) : \delta \tau\in H \Lambda ^{k}(\mathcal T) \; \text{ and }  \; d^\mathcal T \delta \tau = -\tau \} \\
            \mathcal J \Lambda ^k(\mathcal T)  =& \{ \tau \in H\Lambda ^k(\mathcal T) : d^\mathcal T\tau\in H^\star \Lambda ^{k+1}(M) \; \text{ and }  \; \delta d^\mathcal T \tau = -\tau \} \\
          \mathcal J^\star  \Lambda ^{k+1}(\mathcal T)  =& \{ \tau \in H^\star \Lambda ^{k+1}(\mathcal T) : \delta^\mathcal T \tau\in H \Lambda ^{k}(M) \; \text{ and }  \; d \delta^\mathcal T \tau = -\tau \} \\
    \end{align*}
    all hold.
\end{thm}
\begin{proof}
We prove the first identity. The proofs of the other identities are analogous.
Let $\tau \in \mathcal K \Lambda^k(\mathcal T)$. If $\varphi\in C^\infty\Lambda^{k}(M)$ with support in the single simplex $T$, then $\tau$ is orthogonal to $\varphi$.
\[
    0 = \int_T \tau \wedge \star \varphi + d\tau \wedge \star d\varphi 
\]
which means that $\delta d \tau|_T = -\tau|_T $. Conversely, if $\delta d \tau|_T = -\tau|_T$, and $\varphi$ in the kernel of $\ljump \cdot{}^\tangent \rjump$, then
\[
        \langle \tau, \varphi  \rangle _{H } = \int_T \tau \wedge \star \varphi + d\tau \wedge\star d\varphi = \int_T d\varphi\wedge\star d\tau - \varphi \wedge \delta  d \tau  = \ljump \varphi, d\tau \rjump = 0 
    \]
\end{proof}
We will denote the tangential projections {${H\Lambda^k(\mathcal T) \to \mathcal J\Lambda ^k(\mathcal T)}$} and {${H\Lambda^k(M) \to \mathcal K\Lambda ^k(\mathcal T)}$} by $\tau \mapsto \tau ^\tangent$. Similarly the normal projection will be denoted by $\tau ^\normal$. The projections are such that for $\tau \in H\Lambda^k(\mathcal T) $ and $\varphi \in H^\star \Lambda^{k+1}(M)$ we have the identity $\ljump  \tau , \varphi \rjump = \ljump \tau^\tangent , \varphi^\normal \rjump$, which is the left side of the following commuting diagram
\begin{equation*}
    \begin{tikzcd}[column sep = 5em, row sep = 3em]
    H\Lambda ^k(\mathcal T)\times  H^\star \Lambda ^{k+1} (M) \ar[dr, "\ljump \cdot {,}\cdot \rjump", shift left = 0.2em]\ar[d,"\tangent \times \normal",swap] & & H\Lambda ^k(M)\times  H^\star \Lambda ^{k+1} (\mathcal T) \ar[dl, "\ljump \cdot {,}\cdot \rjump",swap, shift right = 0.2em]\ar[d,"\tangent \times \normal"] \\
          \mathcal J \Lambda ^k(\mathcal T) \times  \mathcal K^\star \Lambda ^{k+1} (\mathcal T) \ar[r , "\ljump \cdot {,} \cdot \rjump",swap]  & \mathbb R &  \mathcal K \Lambda ^k(\mathcal T)\times  \mathcal J^\star \Lambda ^{k+1} (\mathcal T)\ar[l , "\ljump \cdot {,} \cdot \rjump"]
    \end{tikzcd}
\end{equation*}
where the down-arrows are the projections. We have a similar identity for the right side. 

Let $\mathring{H}\Lambda^k(T)\subset H\Lambda^k(T)$ be the subspace with vanishing tangential trace and $\mathring{H}^\star\Lambda^{k+1}(T)\subset H^\star \Lambda^{k+1}(T)$ the subspace with vanishing normal trace \parencite[p.~19]{arnold2006finite}.

\begin{thm}\label{thm: Decomposition J Lambda}
    We have the orthogonal decompositions
    \begin{equation*}
        \begin{split}
            H\Lambda^k(\mathcal T) = & \mathcal J \Lambda^k(\mathcal T) \oplus H\Lambda^k(M)\\ = &
            \mathcal J\Lambda^k(\mathcal T) \oplus \mathcal K \Lambda^k(\mathcal T) \oplus  \big( \oplus_{T \in \mathcal T} \mathring{H}\Lambda^k(T) \big )
        \end{split}
    \end{equation*}
    and
    \begin{equation*}
    \begin{split}
          H^\star\Lambda^k(\mathcal T)= & \mathcal J^\star\Lambda^k(\mathcal T) \oplus H^\star\Lambda^k(M) 
              \\ = &
              \mathcal J^\star\Lambda^k(\mathcal T) \oplus \mathcal K ^\star\Lambda^k(\mathcal T) \oplus  \big( \oplus_{T \in \mathcal T} \mathring{H}^\star\Lambda^k(T) \big )  \\ 
    \end{split}
    \end{equation*}
\end{thm}

\begin{prop}\label{prop: Isometries}
    The following maps are all well-defined and isometries
    \begin{equation*}
        \begin{split}
           \mathcal K ^\star \Lambda^{k+1}(\mathcal T) \xrightarrow{\delta} & \mathcal J\Lambda^k(\mathcal T)\\
             \mathcal K \Lambda^{k}(\mathcal T) \xrightarrow{d} & \mathcal J^\star \Lambda^{k+1}(\mathcal T)\\
                  \mathcal J ^\star \Lambda^{k+1}(\mathcal T) \xrightarrow{\delta^\mathcal T} & \mathcal K\Lambda^k(\mathcal T)\\
              \mathcal J \Lambda^{k}(\mathcal T) \xrightarrow{d^\mathcal T}& \mathcal K^\star \Lambda^{k+1}(\mathcal T)
        \end{split}
    \end{equation*}
    Furthermore, we have the identities
    \begin{equation*}
        \begin{aligned}
            \ljump \delta \lambda, \mu\rjump &=  -\langle \lambda, \mu \rangle_{H^\star} \qquad &&\forall \lambda, \mu  \in  \mathcal{K} ^\star \Lambda^{k+1}(\mathcal T) \\
            \ljump \lambda, d\mu\rjump &=  \langle \lambda, \mu \rangle_{H} \qquad &&\forall \lambda, \mu  \in  \mathcal{K}\Lambda^{k}(\mathcal T) \\
            \ljump \delta^\mathcal{T}\lambda, \mu\rjump &=  -\langle \lambda, \mu \rangle_{H^\star} \qquad &&\forall \lambda, \mu  \in  \mathcal{J}^\star \Lambda^{k+1}(\mathcal{T}) \\
             \ljump \lambda, d^\mathcal{T}\mu\rjump &=  \langle \lambda, \mu \rangle_{H} \qquad &&\forall \lambda, \mu  \in  \mathcal{J}\Lambda^{k}(\mathcal T) \\
        \end{aligned}
    \end{equation*} 
\end{prop}
\begin{proof}
We will show the statement for $\delta \from \mathcal{K}^\star \Lambda^{k+1}(\mathcal T)\to \mathcal J\Lambda^k(\mathcal T)$.
Recall that for each $\lambda \in \mathcal K ^\star \Lambda ^{k+1}(\mathcal T)$, we have $ \delta \lambda \in H\Lambda^k(\mathcal T)$ and $\lambda  = - d^\mathcal T \delta \lambda $. Thus
$d^\mathcal{T} \delta \lambda =-\lambda \in H^\star \Lambda^{k+1}(M)$ and $\delta d^\mathcal{T} \delta \lambda=-\delta\lambda$, proving $\delta \lambda \in \mathcal{J}\Lambda^k(\mathcal{T})$.
Furthermore
\[
    \| \delta \lambda \| ^2_{H} = \| \delta \lambda \|^2_{L^2} + \| d^\mathcal T\delta \lambda \| ^2_{L^2}  = \| \lambda\|^2_{L^2} + \| \delta \lambda \| ^2_{L^2}  = \| \lambda \| _{H^\star}^2.
\]
For the identity, let $\lambda, \mu  \in  \mathcal K ^\star \Lambda^{k+1}(\mathcal T)$, then
\[
\ljump \delta \lambda, \mu\rjump= \langle d^\mathcal{T}\delta \lambda, \mu\rangle -  \langle \delta\lambda, \delta\mu\rangle= -\langle\lambda, \mu\rangle-\langle \delta\lambda, \delta\mu\rangle=-\langle\lambda,\mu\rangle_{H^\star}.
\]
Finally, for each $\omega \in \mathcal J\Lambda ^k(\mathcal T)$ we have $-d^\mathcal T\omega \in \mathcal K^\star \Lambda^{k+1} (\mathcal T) $ and $\delta (-d^\mathcal T \omega) = \omega$. This shows surjectivity.
\end{proof}

Later, we want to define elements in jump spaces corresponding to linear forms on $\mathcal{K} ^\star \Lambda^{k+1}(\mathcal T)$ and vice versa. By combining the Riesz representation theorem with 
the preceding proposition, we get the following corollary.
\begin{cor}
    \label{cor: isometry}
    The jump form $\ljump \cdot, \cdot \rjump$ induces isometries
    \begin{align*}
        \left(\mathcal J\Lambda^k(\mathcal T)\right)' &\leftrightarrow \mathcal K^\star \Lambda ^{k+1}(\mathcal T), & \left(\mathcal K^\star \Lambda ^{k+1}(\mathcal T)\right)' &\leftrightarrow \mathcal J\Lambda^k(\mathcal T),\\ 
        \left(\mathcal J^\star\Lambda^{k+1}(\mathcal T)\right)' &\leftrightarrow \mathcal K \Lambda ^{k}(\mathcal T), & \left(\mathcal K \Lambda ^{k}(\mathcal{T})\right)' &\leftrightarrow \mathcal J^\star\Lambda^{k+1}(\mathcal T).\\
    \end{align*}
  
\end{cor}
In particular, for every $F \in (\mathcal J\Lambda^k(\mathcal T))'$ there exists a unique $\lambda \in \mathcal K^\star \Lambda ^{k+1}(\mathcal T)$ such that
\[
        F(\omega) = \ljump \omega , \lambda \rjump \quad \text{for all $\omega \in \mathcal J\Lambda^k(\mathcal T)$,}
\]
and $ \| \lambda \|_{H^\star} = \| F\|_{(\mathcal J\Lambda^k(\mathcal T))'}$.
\begin{definition}
    Let $F \in \big(\mathcal K \Lambda ^{k}(\mathcal{T})\big)'$ and $G \in \big(\mathcal K^\star \Lambda ^{k+1}(\mathcal T)\big)'$. We say any $\phi\in H^\star\Lambda^{k+1}(\mathcal T)$ and $\gamma \in H \Lambda^k(\mathcal T)$ are representatives for $F$ and $G$ respectively, if $\ljump\phi ^\normal \rjump = F$ and $\ljump \gamma^\tangent \rjump = G $.
\end{definition}
By \Cref{cor: isometry} we can always find a representative. 
\begin{lemma}\label{Lem: Harmonic part jump}
    Let $\phi$ and $\tilde{\phi}$ be two representatives of some $F \in \big(\mathcal K \Lambda ^{k}(\mathcal{T})\big)'$. Then $\delta^{\mathcal T}\phi$ and $\delta^{\mathcal T}\tilde \phi$ have the same exact and harmonic part. 
\end{lemma}
\begin{proof}
        We have $\ljump q, \phi - \tilde \phi\rjump =0 $ for all $q \in \mathcal K \Lambda ^{k}(\mathcal{T})$, so $\phi-\tilde{\phi}\in H^\star\Lambda^{k+1}(M)$ by \Cref{lem: nojump K}. Then  $\delta^\mathcal T \phi - \delta^\mathcal T \tilde \phi=\delta (\phi-\tilde{\phi})\in \im \delta \subset L^2\Lambda^k(M)$, and $\delta^\mathcal{T} \phi - \delta^\mathcal T \tilde\phi$ is therefore orthogonal to $\im d\oplus \mathscr{H}\Lambda^ k(M)$.
\end{proof}

\subsection{Lie algebra-valued Sobolev forms.}
We now consider differential forms with values in the Lie algebra $\lie{g}$. All the spaces defined in \Cref{def: Sobolev spaces} have corresponding spaces of Lie algebra-valued forms.

The jumps are defined as follows. Let $\tau \in H\Lambda^k(\mathcal T, \lie g) $ and $\mu \in H^\star\Lambda^{k+1}(M, \lie{g})$, and write 
$\tau=\sum_j \tau^j e_j$, $\mu=\sum_k \mu^k e_k$, where $\{e_i\}_{i=1}^a$ is a basis of $\lie{g}$. Then
\begin{equation}\label{eq: jump lie algebra valued jump}\begin{aligned}
        \ljump \tau , \mu \rjump^{\mathfrak g} =& \int_M \langle d^\mathcal T\tau, \mu \rangle_{x} - \langle \tau, \delta \mu \rangle _x \vol \\ = &
         \sum_{j,k}\langle e_j, e_k \rangle_\lie g \int_M \langle d^\mathcal T\tau^j, \mu^k \rangle_{h} - \langle \tau^j, \delta \mu^k \rangle _h \vol \\ = &
         \sum_{j,k}\langle e_j, e_k \rangle_\lie g \ljump \tau^j , \mu^k\rjump. 
\end{aligned}\end{equation}
Note that the jumps in each component are unrelated. This is in contrast with the twisted jumps in \eqref{eq: jumpdef}. When the structure group $G$ is abelian, the jump constraints are non-twisted, as in \eqref{eq: jump lie algebra valued jump}. We may therefore work componentwise and restrict attention to the case in which the Lie algebra is one-dimensional.

\section{Mixed formulation of the variational problem}\label{sec: Mixed formulation}
We now derive a mixed formulation of the Yang--Mills variational problem. Throughout the remainder of the paper, we work in the following setting:
\begin{itemize}
    \item $M$ is still a connected compact Riemannian manifold without boundary.
    \item The space of harmonic 1-forms on $M$ is trivial, i.e. $\harmonic \Lambda^1(M)= 0$.
    \item The compact Lie group $G$ is abelian and one-dimensional.
    \item There is a fixed triangulation $\mathcal T$ and, for each $T_i\in\mathcal T$, a fixed section $\sigma_i\colon U_i\to P$ on a neighborhood $U_i\supset T_i$.
\end{itemize}
Connectedness is a technical simplification: if $M$ had multiple connected components, we could treat each separately. The assumption $\harmonic \Lambda^1(M)=0$ likewise avoids introducing an additional harmonic variable; filtering out harmonic forms is standard \parencite{arnold2018finite} but is not pursued here.

With appropriate boundary conditions, it would also be possible to handle manifolds with boundary.

The generalization to higher-dimensional abelian Lie groups is straightforward. Non-abelian Lie groups require new methods, which we discuss at the end of the paper.
\todo[inline,disable]{Somewhere in this section, we should write something about what the forms are like in the abelian case, and that the jumps are simple tuples since Ad(g) is the identity matrix, and we can therefore show it only for the case where the lie algebra is U(1)}
We assume that the prescribed jumps $G_\psi$ and $F_\psi$ from \eqref{eq: prescribed jumps} extend to bounded maps $G_\psi \colon \mathcal{K}^\star\Lambda^2(\mathcal{T}) \to \mathbb R $ and $F_\psi \colon \mathcal{K}\Lambda^0(\mathcal{T}) \to \mathbb R$. Concretely, for $\varphi \in C^\infty\Lambda^2(M)$ and $q\in C^\infty\Lambda^0(M) $, we assume that
\[
\begin{aligned}
    |G_\psi(\varphi)| = & 
    \Big | \sum_{\varepsilon_{ji} = 1} \int _f  i_f^\ast(g_{ij}^{-1}dg_{ij}) \wedge i_f^\ast(\star \varphi)  \Big | \leq C\|\varphi\|_{H^\star}\\
     |F_\psi(q)| =& 
     \Big | \sum_{\varepsilon_{ji} = 1} \int _f   i_f^\ast q\wedge i_f^\ast(\star g_{ij}^{-1}dg_{ij})   \Big | \leq C\|q\|_{H}
\end{aligned}
\]
\todo[inline, disable]{Some of this can be moved, maybe even to section 2}

\subsection{Formulation of the variational problem.}
Since we are now working with an abelian Lie group, the Yang--Mills variational problem \eqref{eq: YMgeneral} simplifies to
\begin{equation}\label{eq: YMabelian}
    \begin{aligned}
    &\text{minimize} \frac 1 2  \|d^\mathcal{T}\omega\|^2_{L^2}= \frac 1 2 \sum _{T_i \in \mathcal T}\| d\omega _i \|_{L^2} ^ 2 \\
    &\text{subject to: } \ljump \omega^\tangent \rjump= G_\psi,  \quad \ljump \omega^\normal \rjump= F_\psi
\end{aligned}\end{equation}
By the Hodge decomposition and the assumption $\harmonic\Lambda^1(M)=0$, we can write $\omega = df + \delta g $. Notice that
\[
    \mathcal S(\omega) = \frac 1 2 \| d^\mathcal T\omega\| ^2_{L^2} = \frac 1 2 \| d^\mathcal T \delta g\| ^ 2_{L^2}
\]
and 
\[
    \ljump \omega^\tangent \rjump = \ljump \delta g^\tangent \rjump
\]
only depend on the $\delta g$-component of $\omega$. 

Thus, $\omega$ is not uniquely defined by \eqref{eq: YMabelian}. The remaining gauge freedom in $df$ is fixed by minimizing $\|\delta^\mathcal T df \| ^2_{L^2}$.
We arrive at the following constrained problem:
\begin{equation}\label{eq: YMabelian min delta omega}
    \begin{aligned}
    &\text{minimize} \frac 1 2\|d^\mathcal{T}\omega\|^2_{L^2} + \frac 1 2 \| \delta^\mathcal T\omega\|_{L^2}^2\\
    &\text{subject to: } \ljump \omega^\tangent \rjump= G_\psi,  \quad \ljump \omega^\normal \rjump=  F_\psi.
\end{aligned}\end{equation}
The corresponding Lagrangian is
\begin{equation}\label{eq: Lagrangian min delta omega}
\mathcal L(\omega,p,\lambda) = \frac 1 2  \| d^\mathcal T\omega \| ^2 _{L^2}+ \frac 1 2 \| \delta ^\mathcal T\omega \|_{L^2}^2  + \ljump \omega - \gamma^\tangent, \lambda \rjump + \ljump p , \omega - \phi^\normal \rjump
\end{equation}
where we have picked some representatives $ \phi$ for $ F_\psi$ and $\gamma$ for $G_\psi$.

Currently, $p$ is a Lagrange multiplier enforcing the normal jump constraint. We have some freedom in defining an ambient space for $p$. By \Cref{def: Kdef}, letting $p\in \mathcal K\Lambda^0(\mathcal T)$ would result in a unique $p$. For a simpler numerical implementation, however, we choose $p\in H\Lambda^0(M)$. The normal-jump constraint alone does not determine the components of $p$ outside $\mathcal K\Lambda^0(\mathcal T)$.

The next subsection adds an equation that determines $p\in H\Lambda^0(M)$.

\subsection{Mixed formulation of normal jumps.}

The following two lemmas concern the problem of enforcing normal jump constraints, and simultaneously minimizing $\delta^\mathcal T\omega$, in a mixed formulation. 

Recall that any bounded linear map $F \colon \mathcal K\Lambda ^0(\mathcal T) \to \mathbb R$ can be represented as 
\[\ljump \phi^\normal\rjump\from f\mapsto \ljump f, \phi\rjump\] for a (non-unique) $\phi \in H^\star\Lambda^1(\mathcal{T})$.
\begin{lemma}\label{lem: Jump minimize without p}
    Let $F \colon \mathcal K\Lambda ^0(\mathcal T) \to \mathbb R$ be a bounded linear map, where $F = \ljump \phi^\normal\rjump$ for some $\phi \in H^\star\Lambda^{1}(\mathcal{T})$. The  codifferential of the representative $\phi$ has a Hodge decomposition $\delta^\mathcal T \phi = \delta \beta + h$. If $\omega \in H\Lambda^1(\mathcal T)$ satisfies
    \begin{equation}\label{eq: Weak normal jump without p}
        \langle \omega, dq \rangle _{L^2} - \langle h,q \rangle_{L^2} = \ljump q,\phi^\normal \rjump 
    \end{equation}
    for all $q \in H\Lambda ^0(M)$, then $\ljump v,\omega\rjump = F(v)$ for all $v \in \mathcal K\Lambda^0(\mathcal T)$. Also, among the forms that satisfy $ \ljump \omega^\normal \rjump = F$, $\omega$ is the one which minimizes $\|\delta^\mathcal T\omega\|^2_{L^2}$.
We also have $\delta^\mathcal T \omega = h$.
\end{lemma}
\begin{proof}
Expanding \eqref{eq: Weak normal jump without p} gives
\begin{equation*}
         \langle \omega, dq \rangle _{L^2} - \langle h,q \rangle_{L^2} = \ljump q,\phi^\normal \rjump  = 
         \langle \phi,dq \rangle _{L^2} - \langle q , \delta \beta \rangle _{L^2} - \langle h,q \rangle _{L^2}
\end{equation*}
and simplifying gives
\begin{equation*}\begin{aligned}
         \langle \omega - \phi, dq \rangle _{L^2}  =\langle q ,- \delta \beta \rangle _{L^2}
\end{aligned}\end{equation*}
for all $q \in H\Lambda ^0(M)$, which is equivalent to $\delta(\omega - \phi ) = - \delta \beta$. This means that $\omega - \phi \in H^\star\Lambda^1(M)$, which by \Cref{lem: nojump} gives $\ljump (\omega  - \phi)^\normal\rjump = 0$ and proves the first part. We could instead choose any $\alpha\in\im\delta$ on the right-hand side of $\delta(\omega-\phi)=\alpha$. Hodge orthogonality then gives
\begin{equation*}
    \| \delta ^\mathcal T\omega \| ^2_{L^2} = \|  \alpha + \delta^\mathcal T\phi \| ^2_{L^2} = \| \alpha + \delta \beta + h \| ^2 _{L^2} = \| \alpha + \delta \beta \| ^2_{L^2} + \| h \| ^2 _{L^2}
\end{equation*}
which is minimized at $\alpha = -\delta \beta$, confirming the second part. Lastly moving the terms around in $\delta(\omega - \phi) = -\delta \beta $ and using $\delta^{\mathcal T}\phi = \delta\beta +h$ we get $\delta^\mathcal T\omega = h$. 
\end{proof}
By \Cref{Lem: Harmonic part jump}, $h$ depends only on $F_\psi$, not on the representative $\phi$. Instead of taking the harmonic part of $\delta^\mathcal T \phi$ directly, we use the minimizer $p$ of $\| p - \delta^\mathcal T  \phi\|^2_{L^2}$ subject to $\Delta p = \delta dp = 0$. This gives a new version of \Cref{lem: Jump minimize without p} in which $p = h$.
\begin{lemma}[Weak mixed normal jump]\label{lem: Jump minimize with p}
     Let $F \colon \mathcal K\Lambda ^0(\mathcal T) \to \mathbb R$ be a bounded linear map, and $F = \ljump \phi^\normal \rjump$ for some $\phi \in H^\star\Lambda^{1}(\mathcal{T})$. The  codifferential of the representative $\phi$ has Hodge decomposition $\delta^\mathcal T \phi= \delta \beta + h$. If $\omega \in H\Lambda^1(\mathcal T)$ and $p \in H\Lambda^0(M)$ satisfy
    \begin{align}\label{eq: Weak normal jump with p}
        \langle \omega, dq \rangle _{L^2} - \langle p,q \rangle_{L^2} = \ljump q,\phi^\normal \rjump &,&  \langle dp , d\varphi \rangle_{L^2} = 0
    \end{align}
    for all $q \in H\Lambda ^0(M)$ and $\varphi \in H\Lambda^0(M)$, then $\ljump q,\omega\rjump = F(q)$ for all $q \in \mathcal K\Lambda^0(\mathcal T)$. Also, among the forms that satisfy $ \ljump \omega^\normal \rjump = F$, $\omega$ is the one which minimizes $\|\delta^\mathcal T\omega\|^2_{L^2}$.
We also have $\delta^\mathcal T \omega = h = p$.
\end{lemma}
\begin{proof}
    We need to show that $p=h$, so that the result follows from \Cref{lem: Jump minimize without p}. The second equation is equivalent to $\Delta p=0$, so $p$ is harmonic. If $q = \harmonic q$ is harmonic, the orthogonality of the Hodge decomposition gives
    \[
        -\langle \harmonic q , p \rangle _{L^2} = \langle d\harmonic q, \phi \rangle _{L^2} - \langle\harmonic q, \delta^\mathcal T\phi \rangle _{L^2}  = -\langle\harmonic q, h \rangle _{L^2}
    \]
    which implies that $p = h$. 
\end{proof}
\subsection{Infinite-dimensional saddle-point problem.}
We now use \Cref{lem: Jump minimize with p} to derive variational equations for the stationary points of the Lagrangian $\mathcal L$ in \eqref{eq: Lagrangian min delta omega}.
\setlength{\forallsep}{0cm}
\begin{thm}\label{thm: Variational problem T}
    Let $\phi$ and $\gamma$ be any representatives for $F_\psi$ and $G_\psi$. Solutions $(\omega,p,\lambda)\in H\Lambda^1(\mathcal{T})\oplus H\Lambda^{0}(M)\oplus \mathcal{K}^\star \Lambda^2(\mathcal{T})$
    to the variational problem
    \begin{subequations}\label{eq: variational problem}
    \begin{align}
    \label{eq: variational problem1}
    \langle \omega , dq \rangle_{L^2}  - \langle p,q \rangle_{L^2}&= \ljump q, \phi^\normal\rjump 
    & \forallspace &\forall q\in H\Lambda^{0}(M)\\ \label{eq: variational problem2}
    \langle d^\mathcal T \omega , d^\mathcal T \tau \rangle _{L^2} + \langle dp ,  \tau  \rangle_{L^2} + \ljump \tau, \lambda \rjump&=0
    & \forallspace&\forall \tau \in H\Lambda^1(\mathcal{T})\\ \label{eq: variational problem3}
    \ljump \omega , \mu \rjump &= \ljump \gamma^\tangent, \mu\rjump
    & \forallspace&\forall \mu \in \mathcal{K}^\star \Lambda^2(\mathcal{T})
    \end{align}
    \end{subequations}
    are solutions to the Yang--Mills variational problem \eqref{eq: YMabelian min delta omega}.
\end{thm}
\begin{proof}
    We use the same representatives $\phi$ and $\gamma$ in the Lagrangian. Let $(\omega, p, \lambda)$ be a solution. First, \eqref{eq: variational problem2} implies that
    \[
        \langle dp, d\varphi \rangle  = 0
    \]
    for all $\varphi \in H\Lambda ^0(M)$. This and \eqref{eq: variational problem1} allow us to apply \Cref{lem: Jump minimize with p}, which tells us that $\delta^\mathcal T\omega = p$ and $\ljump q, \omega - \phi^\normal\rjump =0$. These two identities combined with \eqref{eq: variational problem2} and \eqref{eq: variational problem3} show that $(\omega, p, \lambda)$ is a stationary point of the Lagrangian $\mathcal L$ in \eqref{eq: Lagrangian min delta omega}:
    \begin{equation*}\begin{aligned}
        \delta \mathcal L(\omega, p , \lambda \, ;  \tau, q, \mu) = &  \langle d^\mathcal T\omega,d^\mathcal T \tau \rangle + \langle \delta^\mathcal T\omega, \delta^\mathcal T\tau\rangle \\ & + \ljump \omega - \gamma^\tangent , \mu \rjump  + \ljump q, \omega - \phi^\normal \rjump + \ljump \tau , \lambda \rjump + \ljump p, \tau \rjump  \\ = &
     \langle d^\mathcal T\omega,d^\mathcal T\tau \rangle + \langle \delta^\mathcal T\omega, \delta^\mathcal T\tau\rangle \\ & + \ljump \omega - \gamma^\tangent , \mu \rjump  + \ljump q, \omega - \phi^\normal \rjump + \ljump \tau , \lambda \rjump + \langle dp, \tau \rangle - \langle p, \delta ^\mathcal T \tau \rangle  \\ = &
      \big ( \langle d^\mathcal T \omega,d^\mathcal T \tau \rangle + \langle dp, \tau \rangle + \ljump \tau , \lambda \rjump \big )  + \ljump \omega - \gamma^\tangent , \mu \rjump  + \ljump q, \omega - \phi^\normal \rjump\\ =&  0
    \end{aligned}\end{equation*}
\end{proof}

\subsection{Well-posedness of the variational problem.}
We now proceed to show that the variational problem \eqref{eq: variational problem} is well-posed.
To simplify notation, define
$A\from \left(H\Lambda^1(\mathcal{T})\oplus H\Lambda^0(M)\right) \times \left(H\Lambda^1(\mathcal{T})\oplus H\Lambda^0(M)\right)\to \RR$
by
\begin{equation*}\label{Matrix A}
    A(\omega, p \mid \tau, q )  = \langle d^\mathcal{T} \omega , d^\mathcal{T} \tau \rangle  + \langle dq , \omega \rangle + \langle dp, \tau \rangle - \langle p, q \rangle 
\end{equation*}
and 
$B\from H\Lambda^1(\mathcal{T}) \times \mathcal{K}^\star \Lambda^2(\mathcal{T})\to \RR$ by
\[
    B(\omega\mid \mu) = \ljump \omega , \mu \rjump 
\]
\setlength{\forallsep}{1cm}
\begin{thm}\label{thm: Well posed A + B}
    Let $F\colon H\Lambda^1(\mathcal T) \oplus H\Lambda^0(M) \to \mathbb R $ and $G\colon \mathcal K ^\star \Lambda^2(\mathcal T)\to \mathbb R$ be any bounded linear operators, and $A$ and $B$ as above. Then the saddle-point problem
    \begin{equation}\label{eq: Saddle-point problem}
        \begin{aligned}
        A(\omega, p \mid \tau , q) +B(\tau \mid   \lambda) = &
        F(\tau,q ) &\forallspace &\forall (\tau, q) \in  H\Lambda^1(\mathcal T)\mathrlap{\oplus H\Lambda^0(M)} \\
        B(\omega \mid \mu ) =&  G(\mu) &\forallspace& \forall \mu \in  \mathcal K^\star\Lambda^2(\mathcal T)
        \end{aligned}
    \end{equation}
    is well-posed. 
\end{thm}
\begin{proof}
    A standard FEEC result shows that the bilinear form $A$ satisfies the inf-sup condition when $\omega , \tau \in H\Lambda^1(M) = \ker {\ljump \cdot {}^\tangent\rjump}$ \parencite[Theorem~4.9]{arnold2018finite}. This space is $\ker B$ in standard saddle-point notation \parencite{boffi2013mixed}.
    This verifies one of the two conditions of the relevant saddle-point theorem \parencite[Theorem~4.2.3]{boffi2013mixed}. We will show that the second one
    \begin{equation*}\label{infsup B}
        \inf_{\substack{\lambda \in \mathcal K^\star \Lambda^2(\mathcal T) \\ \| \lambda\|_{H^\star} = 1}} \sup_{\substack{\omega\in H\Lambda^1(\mathcal T)\\ \|\omega\|_H = 1}} B(\omega\mid \lambda) > 0
    \end{equation*}
    is also satisfied. Using \Cref{cor: isometry} we have
    \[
        \sup_{\substack{\omega\in H\Lambda^1(\mathcal T)\\ \|\omega\|_H = 1}} \ljump \omega, \lambda \rjump \geq     \sup_{\substack{\omega\in \mathcal J\Lambda^1(\mathcal T)\\ \|\omega\|_H = 1}} \ljump \omega, \lambda \rjump = \| \lambda\|_{H^\star}
    \]
    and then since $B(\omega\mid \lambda) = \ljump \omega, \lambda\rjump $ we have
    \[
       \inf_{\substack{\lambda \in \mathcal K^\star \Lambda^2(\mathcal T) \\ \| \lambda\|_{H^\star} = 1}}\sup_{\substack{\omega\in H\Lambda^1(\mathcal T)\\ \|\omega\|_H = 1}}  B(\omega\mid \lambda) \geq \inf_{\substack{\lambda \in \mathcal K^\star \Lambda^2(\mathcal T) \\ \| \lambda\|_{H^\star} = 1}} \| \lambda \|_{H^\star} = 1 \, .
    \]
    Since the two conditions of the saddle-point theorem are satisfied, the problem is well-posed \parencite[Theorem~4.2.3]{boffi2013mixed}.
\end{proof}
\begin{remark}
    The theorem above also holds if we replace $\mathcal T$ with $\mathcal T_h$, and $\ljump \cdot, \cdot \rjump$ with $\ljump \cdot, \cdot \rjump_{\mathcal T_h}$.
\end{remark}
If $F(\tau, q)=F_\psi(q)$, and $G(\mu)=G_\psi(\mu)$, we recover \eqref{eq: variational problem}, and from now on let  $F = F_\psi$ and $G = G_\psi$.
\subsection{Refined triangulation variational problem.}
\todo[inline,disable]{This section is in some sense contrary to the the statement at the beginning that there is a fixed triangulation with fixed sections. The point is that the triangulation serves two roles: 1. Defining the connection form by local sections 2.Defining the finite element spaces.
Typically we want to refine the finite element spaces, without changing the local sections. I don't think we handle this very well at present. We just write 'if the triangulation does not introduce any new jumps' or similar. I think we should be clearer and more careful here - Charles}

\setlength{\forallsep}{0cm}
\begin{prop}\label{prop: saddle point T_h}
    Suppose we replace $d^{\mathcal T}$ with $d^{\mathcal T_h}$ and $\ljump \cdot , \cdot \rjump$ with $\ljump \cdot , \cdot \rjump_{\mathcal T_h}$ in \eqref{eq: variational problem}, and extend the jump functional to the refined multiplier space by setting $G(\mu)=\ljump\gamma,\mu\rjump_{\mathcal T_h}$. Then we have another saddle point problem
     \begin{subequations}\label{eq: Saddle-point problem T_h}
        \begin{align}
        A_h(\omega, p \mid \tau , q) +B_h(\tau \mid   \lambda) = &
        F(q) &\forallspace &\forall (\tau, q) \in  H\Lambda^1(\mathcal T_h)\mathrlap{\oplus H\Lambda^0(M)} \label{eq: Saddle-point problem T_h 1}\\
        B_h(\omega \mid \mu ) =&  G(\mu) &\forallspace& \forall \mu \in  \mathcal K^\star\Lambda^2(\mathcal T_h)\label{eq: Saddle-point problem T_h 2}
        \end{align}
    \end{subequations}
    If $(\omega, p, \lambda)$ is the solution of \eqref{eq: variational problem}, or equivalently of \eqref{eq: Saddle-point problem T_h} with $\mathcal T_h = \mathcal T$, and $(\omega^h, p^h, \lambda^h)$ is a solution of \eqref{eq: Saddle-point problem T_h} for a refinement $\mathcal T_h$ of $\mathcal T$, then $(\omega^h, p^h, (\lambda^h)^\normal) = (\omega, p, \lambda)$, where $\normal$ denotes the projection $\normal \colon H^\star\Lambda^2(M) \to \mathcal K^\star\Lambda^2(\mathcal T) $.
\end{prop}
\begin{proof}
    Suppose $(\omega^h, p^h,\lambda^h)$ is a solution to \eqref{eq: Saddle-point problem T_h} with $\mathcal T_h$ a refinement of $\mathcal T$. Then \eqref{eq: Saddle-point problem T_h 2} implies by \Cref{lem: No refined jump}, that $\omega^h \in H\Lambda^1(\mathcal T)$. It is then clear that $(\omega^h, p^h, \lambda^h)$ solves \eqref{eq: Saddle-point problem T_h 1} with $\mathcal T_h = \mathcal T$. Taking the difference of \eqref{eq: variational problem2} for the two solutions, we have
    \[
        \ljump \tau, \lambda^h - \lambda \rjump = 0
    \]
    for all $\tau \in H\Lambda^1(\mathcal T)$, which is equivalent to $(\lambda^h)^\normal = \lambda$. It is now clear that $(\omega^h,p^h,(\lambda^h)^\normal)$ is a solution to \eqref{eq: Saddle-point problem T_h} with $\mathcal T_h = \mathcal T$, and uniqueness gives $(\omega^h, p^h, (\lambda^h)^\normal) = (\omega, p, \lambda)$.
\end{proof}

\section{Finite element subspaces}

For finite element approximations of the Yang--Mills variational problem, it is useful to employ spaces from finite element exterior calculus \parencite{arnold2018finite}. For a thorough introduction, see the standard references \parencite{arnold2006finite,arnold2018finite}.

We will use the ``trimmed'' polynomial spaces over a simplex $T$,  $\mathcal{P}_r^- \Lambda^{k}(T)$, where $k$ is the order of the differential forms, and $r$ is a polynomial degree. 

To illustrate these spaces, \Cref{tab: FE} shows bases when $T=\Delta_2$ is the standard simplex in $\RR^2.$

\begin{table}[ht]
\begin{tabular}{c|p{8em}p{11em}p{11em}}
  & $k=0$ & $k=1$ & $k=2$\\
  \hline
  $r=1$ &  $1,x,y$& $dx,dy,xdy-ydx$ & $dx\wedge dy$\\
  $r=2$ & $1,x,y,x^2,xy,y^2$ & $dx,dy,xdx,ydx$,$ydy$,$xdy$, $x(xdy-ydx),y(xdy-ydx)$ & $dx\wedge dy,xdx\wedge dy,$ $ydx\wedge dy$\\
\end{tabular}
\caption{Bases for $\mathcal{P}_r^- \Lambda^{k}(\Delta_2)$}
\label{tab: FE}
\end{table}
When $T$ is a two-dimensional simplex on a manifold $M$, or is embedded in $\RR^3$, a diffeomorphism $\phi\from \Delta_2\to T$ can be used to push-forward the space from $\Delta_2$ to $T$.
The advantage of these spaces is that they interact naturally with the exterior derivative. More precisely, they form the chain complex
\[
0\rightarrow \RR \rightarrow \mathcal{P}^-_r\Lambda^0(T) \xrightarrow{d} \mathcal{P}^-_r\Lambda^1(T) \xrightarrow{d} \mathcal{P}^-_r\Lambda^2(T) \rightarrow 0
\]
\subsection{Traces of polynomial forms in two dimensions.}
Let $T$ be the standard reference triangle in $\mathbb R^2$, and let $i\colon \partial T \to T$ be the inclusion. Let $\partial T = f_0\cup f_1 \cup f_2$, where the $f_j$ are facets. The basis
\begin{align*}
    \phi_0  =(1-y)dx +xdy && \phi_1 = ydx + (1-x)dy && \phi_2 = -ydx + xdy
\end{align*}
of $\mathcal P _1^-\Lambda^1(T)$ is such that
\begin{align*}
    i_{f_0}^\ast\phi_0 = \vol_{f_0} && i_{f_1}^\ast\phi_1 = \vol_{f_1} &&  i_{f_2}^\ast\phi_2 = \vol_{f_2} &&
\end{align*}
and in general $i^\ast_{f_i} \phi_j = \delta_{ij}\vol_{f_j}$.
This allows us to prove the following.
\begin{lemma}\label{Polynomial trace}
  Let $T$ and $i$ be as above. Then for each $r\geq 1$, the map $i^\ast\colon  \mathcal P_{r+1}^-\Lambda^1(T) \to \bigoplus_j \mathcal P_{r}\Lambda^1(f_j)$ is surjective. 
\end{lemma}
\begin{proof}
Let $\varphi_j=p_j\vol_{f_j}\in\mathcal P_r\Lambda^1(f_j)$, and choose an extension $\widetilde p_j\in\mathcal P_r(T)$ such that $i_{f_j}^\ast\widetilde p_j=p_j$. Set $\eta=\kappa(dx\wedge dy)$. Since
\[
    \phi_0=dx+\eta, \qquad \phi_1=dy-\eta, \qquad \phi_2=\eta,
\]
we have
\[
    \widetilde p_j\phi_j\in
    \mathcal P_r\Lambda^1(T)+\kappa\bigl(\mathcal P_r\Lambda^2(T)\bigr)
    =\mathcal P_{r+1}^-\Lambda^1(T)
\]
for each $j$. Therefore $\omega=\sum_{j=0}^2\widetilde p_j\phi_j$ belongs to $\mathcal P_{r+1}^-\Lambda^1(T)$ and satisfies
\[
    i^\ast\omega=(\varphi_0,\varphi_1,\varphi_2).
\]
\end{proof}
\subsection{Jumps for finite element forms.}
This subsection is the finite element counterpart of \Cref{sec: Function spaces on the base manifold}, where we defined analogous spaces and obtained analogous results. Given a triangulation $\mathcal T_h$ of $M$, we construct the discontinuous finite element space
\[
\begin{aligned}
D\mathcal{P}^-_r\Lambda^k(\mathcal T_h)&= \bigoplus_{T \in \mathcal T_h} \mathcal{P}^-_r\Lambda^k(T) \subset H\Lambda^k(\mathcal T_h)
\end{aligned}
\]
A natural question is when a finite element form $\tau_h \in D\mathcal{P}^-_r\Lambda^1(\mathcal T_h)$ is contained in the continuous finite element space $\mathcal P_r^-\Lambda^1(\mathcal T_h)$. Define the skeleton space
\[
    V\coloneqq\bigoplus_{f\in\mathcal F_h}L^2\Lambda^1(f)
\]
with its natural $L^2$-inner product, and define $Q_h\colon D\mathcal P_r^-\Lambda^1(\mathcal T_h)\to V$ by
\[
    Q_h(\tau_h) = \bigoplus_{\varepsilon _{ji}= 1}\left(i_f^\ast(\tau_h|_{T_j}) - i_f^\ast(\tau_h|_{T_i})\right).
\]
Also, let $P_{\mathcal J_h}\colon H\Lambda^1(\mathcal T_h) \to \mathcal J\Lambda^1(\mathcal T_h)$ be the orthogonal projection. We have the following characterization.
\begin{prop}\label{prop: jumps polynomial}
    Let $\tau_h \in D\mathcal{P}^-_r\Lambda^1(\mathcal T_h)$. Then the following are equivalent:
    \begin{enumerate}[label=\textup{\roman*.}, ref=\textup{\roman*}]
        \item $\tau_h \in \mathcal P^-_r\Lambda^1(\mathcal T_h)$ \label{item: P}
        \item $\tau_h \in H\Lambda^1(M)$ \label{item: H}
        \item For each $ f = T_i\cap T_j  \in \mathcal F_h$ we have $i^\ast_f (\tau_h|_{T_j})- i_f^\ast (\tau_h|_{T_i}) = 0$ \label{item: i*}
        \item For each $f = T_i\cap T_j\in \mathcal F_h, d \geq 1, g \in \Delta_d(f), \varphi \in \mathcal P_{r-d}\Lambda^{d-1}( g)$ we have \[\int_g \big( i^\ast_g(\tau_h|_{T_j})- i_g^\ast (\tau_h|_{T_i})\big)\wedge \varphi =0 \] \label{item: DOF jump}
        \item $Q_h(\tau_h) =0$ \label{item: Q}
        \item For each $\mu \in \mathcal K^\star \Lambda^2(\mathcal T_h) $ we have $\ljump \tau_h , \mu \rjump_{\mathcal T_h} =0 $.\label{item: cont jump point}
        \item $P_{\mathcal J_h}\tau_h = 0 $ \label{item: Projection}
    \end{enumerate}
\end{prop}
\newcommand{\jumpref}[1]{\hyperref[#1]{\Cref*{prop: jumps polynomial}.\textit{\labelcref*{#1}}}}
\begin{proof}
    A standard FEEC result \parencite[Theorem~5.1]{arnold2006finite} gives equivalences \labelcref*{item: P}.--\labelcref*{item: i*}., and  \labelcref*{item: i*}.$\Longleftrightarrow$\labelcref*{item: DOF jump}. comes from the fact that the DOFs of each $\mathcal P_r^-\Lambda^1(f)$ are unisolvent. Next, \labelcref*{item: Q}. is just  \labelcref*{item: i*}. where we look at all $f \in \mathcal F_h$ together. Lastly, since $D\mathcal P_r^- \Lambda^1(\mathcal T_h) \subset H\Lambda^1(\mathcal T_h)$ we get equivalences \labelcref*{item: H}.$\Longleftrightarrow$\labelcref*{item: cont jump point}. $\Longleftrightarrow$ \labelcref*{item: Projection}. by \Cref{lem: nojump K} and \Cref{cor: isometry}
\end{proof}

The characterizations serve different purposes. For the analysis of the finite element solution, \jumpref{item: cont jump point} is the most natural; for implementation, \jumpref{item: Q} is the most practical; and for jumps of interpolants $\Pi_{\mathcal T}^h\tau$, \jumpref{item: DOF jump} is the most useful.

Notice that $P_{\mathcal J_h}|_{D\mathcal P_r^-\Lambda^1(\mathcal T_h)}\colon D\mathcal P_r^-\Lambda^1(\mathcal T_h) \to \mathcal J\Lambda^1(\mathcal T_h)$ and $Q_h\colon D\mathcal P_r^-\Lambda^1(\mathcal T_h)\to V$ have the same kernel, namely $\mathcal P_r^-\Lambda^1(\mathcal T_h)$. Taking its orthogonal complement in $D\mathcal P_r^-\Lambda^1(\mathcal T_h)$ defines the space $J_h$. Let $\mathcal B_h$ and $V_h$ be the images of $P_{\mathcal J_h}$ and $Q_h$, respectively. Then we have induced isomorphisms
\begin{equation}\label{eq: isomorphisms V_h J_h B_h}
\begin{tikzcd}
        V_h & \ar[l, "Q_h",swap, "\sim"'] J_h \ar[r, "P_{\mathcal J_h}", "\sim"'] & \mathcal B_h
\end{tikzcd}
\end{equation}
To enforce the equivalent constraints $Q_h(\tau_h)=0$ or $P_{\mathcal J_h}(\tau_h)=0$, we could use Lagrange multipliers in $V_h$ and $\mathcal B_h$, respectively. Instead of using $\mathcal B_h$ as the second multiplier space, we use its isometric image $K_h^\star \coloneq d^{\mathcal T_h} \mathcal B_h \subset \mathcal K^\star\Lambda^2(\mathcal T_h)$ from \Cref{prop: Isometries}. For $Q_h(\tau_h)$, we use $\langle Q_h(\tau_h),v_h\rangle_V$.

We construct an isomorphism $R_h\colon K_h^\star\to V_h$. Let $\mu_h\in K_h^\star$. The map $\ljump Q_h^{-1}(\,\cdot\,),\mu_h\rjump_{\mathcal T_h}$ is an element of $(V_h)'$. The Riesz representation theorem therefore gives $R_h\mu_h\in V_h$ satisfying
\[
    \langle v_h, R_h \mu_h \rangle _V = \ljump Q_h^{-1}(v_h) , \mu_h \rjump_{\mathcal T_h}
\]
for all $v_h \in V_h$. This is linear because both Riesz and the jump form are linear. We automatically have
\[
 \langle Q_h\tau_h, R_h \mu_h\rangle_V= \ljump \tau_h , \mu_h \rjump_{\mathcal T_h} 
\]
for all $\tau_h \in J_h$. The properties of $R_h$ imply that $R_h$ is an isomorphism.  Since we have both sides equal to zero for $\tau_h \in \mathcal P_r^-\Lambda^1(\mathcal T_h)$, the identity above also holds for each $\tau_h \in D\mathcal P_r^-\Lambda^1(\mathcal T_h)$.
\subsection{Jump projections.}
\begin{lemma}\label{lem: Projection right}
    There exists a map $P_{K_h^\star}\colon  \mathcal K^\star \Lambda^2(\mathcal T_h) \to K_h^\star$ such that
    \[
        \ljump \tau_h  , P_{K_h^\star}\mu  - \mu \rjump_{\mathcal T_h} = 0
    \]
    for all $\tau_h \in D\mathcal P_r^-\Lambda^1(\mathcal T_h)$. 
\end{lemma}
\begin{proof}
    The element $\mu \in \mathcal K^\star\Lambda^2(\mathcal T_h)$ defines a linear map $ \ljump \cdot , \mu \rjump \colon \mathcal B_h \to \mathbb R$. By Riesz and \Cref{cor: isometry}, there exists some $P_{K_h^\star}\mu \in K_h^\star$ such that $\ljump \tau_h, P_{K_h^\star}\mu \rjump_{\mathcal T_h} = \ljump \tau_h , \mu \rjump_{\mathcal T_h}$ for all $\tau_h \in \mathcal B_h$. Since $\tau_h-P_{\mathcal J_h}\tau_h\in H\Lambda^1(M)$, we have
    \[
        \ljump \tau_h, P_{K_h^\star}\mu \rjump_{\mathcal T_h} =\ljump P_{\mathcal J_h}\tau_h, P_{K_h^\star}\mu \rjump_{\mathcal T_h} =\ljump P_{\mathcal J_h}\tau_h, \mu \rjump_{\mathcal T_h}= \ljump \tau_h , \mu \rjump_{\mathcal T_h}
    \]
    for all $\tau_h \in D\mathcal P_r^-\Lambda^1(\mathcal T_h)$.
\end{proof}
\begin{cor}\label{cor: discrete kernel}
    Let $\tau_h \in D\mathcal P_r^-\Lambda^1(\mathcal T_h)$. If 
    \[
        \ljump \tau_h , \mu_h \rjump_{\mathcal T_h} = 0 
    \]
    for all $\mu_h \in K_h^\star$, then $\tau_h \in \mathcal P_r^-\Lambda^1(\mathcal T_h)$
\end{cor}
\begin{proof}
    Suppose
    \[
        \ljump \tau_h , \mu_h \rjump_{\mathcal T_h} = 0 
    \]
    for all $\mu_h \in K_h^\star$. Then for $\mu \in \mathcal K^\star\Lambda^2(\mathcal T_h)$ we have
     \[
        \ljump \tau_h , \mu\rjump_{\mathcal T_h}  = \ljump \tau_h , P_{K_h^\star}\mu\rjump_{\mathcal T_h}  = 0 
    \]
    and therefore $\tau_h \in \mathcal P_r^-\Lambda^1(\mathcal T_h)$ by \Cref{prop: jumps polynomial}. 
\end{proof}

\begin{lemma}\label{lem: Projection left}
    There exists a map $P_{J_h}\colon H\Lambda^1(\mathcal T_h) \to J_h$ such that
    \[
 \ljump P_{J_h} \tau - \tau, \mu_h \rjump_{\mathcal T_h} =0
    \]
    for all $\tau \in H\Lambda^1(\mathcal T_h)$ and $\mu_h \in K_h^\star$.
\end{lemma}
\begin{proof}
    Let $P_{\mathcal B_h}\tau$ denote the orthogonal projection of $\tau$ onto $\mathcal B_h$. Then we can use the isomorphism \eqref{eq: isomorphisms V_h J_h B_h} to find the element $P_{J_h}\tau$ corresponding to $P_{\mathcal B_h}\tau$. By orthogonality, we have
    \[
        \ljump \tau, \mu_h \rjump_{\mathcal T_h} = -\langle \tau , \delta \mu_h \rangle_H  = -\langle P_{\mathcal B_h}\tau , \delta \mu_h \rangle _H =\ljump P_{J_h} \tau , \mu_h \rjump_{\mathcal T_h}
    \]
\end{proof}
\subsection{Inf-sup inequality.}
\begin{lemma}\label{lem: inf-sup finite element}
    We have the inf-sup inequality
    \[
     \beta_h=  \inf_{\substack{\lambda_h \neq 0\\ \lambda_h \in K_h^\star }}\sup_{\substack{\tau_h \neq 0 \\ \tau_h \in D \mathcal  P_r^-\Lambda^1(\mathcal T_h)}}\frac{|\ljump \tau_h,\lambda_h \rjump_{\mathcal T_h}| }{\| \lambda_h\|_{H^\star} \|\tau_h\|_H} >0
    \]
\end{lemma}
\begin{proof}
    $J_h \subset D\mathcal P_r^-\Lambda^1(\mathcal T_h)$ is the orthogonal complement of the left kernel of $\ljump\cdot, \cdot \rjump$.
    
    Thus, for $\lambda_h \in K_h^\star$ and $\tau_h\in D\mathcal P_r^-\Lambda^1(\mathcal T_h)$, we can write $\tau_h= \widetilde{\tau}_h + \sigma_h$, where $\widetilde{\tau}_h\in J_h$ and $\sigma_h\in \mathcal P_r^-\Lambda^1(\mathcal T_h)$.
    Then 
    \begin{equation}
        \label{eq: lemma57ineq}
    \frac{|\ljump \tau_h,\lambda_h \rjump_{\mathcal T_h}| }{ \|\tau_h\|_H}= \frac{|\ljump \widetilde{\tau}_h, \lambda_h\rjump_{\mathcal T_h}|}{\sqrt{\|\widetilde{\tau}_h\|_H^2+ \|\sigma_h\|_H^2}}\le \frac{|\ljump \widetilde{\tau}_h,\lambda_h \rjump_{\mathcal T_h}| }{ \|\widetilde{\tau}_h\|_H}
    \end{equation}

    For every $\lambda_h \in K_h^\star$ we have that
    \begin{equation}
    \begin{aligned}
     \sup_{\substack{\tau_h \neq 0 \\ \tau_h \in D\mathcal P_r^-\Lambda^1(\mathcal T_h)}}\frac{|\ljump \tau_h,\lambda_h \rjump_{\mathcal T_h} |}{\| \lambda_h\|_{H^\star} \|\tau_h\|_H}
      = &\sup_{\substack{\tau_h \neq 0 \\ \tau_h \in J_h}}\frac{|\ljump \tau_h,\lambda_h \rjump_{\mathcal T_h} |}{\| \lambda_h\|_{H^\star} \|\tau_h\|_H}
      \\ 
     = &\sup_{\substack{P_{J_h}\tau \neq 0 \\ \tau \in H\Lambda^1(\mathcal T_h)}}\frac{|\ljump P_{J_h}\tau,\lambda_h \rjump_{\mathcal T_h}| }{\|\lambda_h\|_{H^\star} \|P_{J_h}\tau\|_H} \\  \geq   & 
      \sup_{\substack{P_{J_h}\tau \neq 0 \\ \tau \in H\Lambda^1(\mathcal T_h)}}\frac{|\ljump P_{J_h}\tau,\lambda_h \rjump_{\mathcal T_h}| }{\|\lambda_h \|_{H^\star} \| P_{J_h}\|\|\tau\|_H} 
      \\ = & 
     \sup_{\substack{\tau \neq 0 \\ \tau \in H\Lambda^1(\mathcal T_h)}}\frac{|\ljump \tau,\lambda_h \rjump_{\mathcal T_h}| }{\|\lambda_h \|_{H^\star} \| P_{J_h}\|\|\tau\|_H} 
      \\ =& \frac 1 {\|P_{J_h}\|}
    \end{aligned}
    \end{equation}
    where the first equality follows from \eqref{eq: lemma57ineq}, the equality after the inequality from \Cref{lem: Projection left}, and the final equality from \Cref{cor: isometry}.
    The inf-sup inequality is therefore satisfied, with $\beta_h \geq \frac 1 {\| P_{J_h}\|}>0$.
\end{proof}

\section{Well-posedness of the finite element problem}

For the analysis of the finite element method, it is convenient to always consider possible refinements to a triangulation. By \Cref{prop: saddle point T_h}, solutions $(\omega, p , \lambda)$ of \eqref{eq: variational problem} can always be obtained from the solution $(\omega^h, p^h,\lambda^h)$ of \eqref{eq: Saddle-point problem T_h} with a refinement $\mathcal T_h$ of $\mathcal T$. We will therefore approximate \eqref{eq: Saddle-point problem T_h} with finite elements.
\subsection{Finite element problem.}
For the finite element approximation of our problem, we use the subspaces
\begin{align*}
    D\mathcal P_r ^-\Lambda^1(\mathcal T_h) \subset H\Lambda^1(\mathcal T_h) && \mathcal P_r^-\Lambda^0(\mathcal T_h) \subset H\Lambda ^0 (M) && K_h^\star \subset \mathcal  K^\star\Lambda^2(\mathcal T_h)
\end{align*}
where ``$D$'' means the discontinuous/broken version.

\begin{thm}
    \label{thm: altformulation}
    For every pair of bounded functionals $F\colon H\Lambda^0(M)\to\mathbb R$ and $G_h\colon K_h^\star\to\mathbb R$, the problem of finding $(\omega_h,p_h,\lambda_h)\in D\mathcal P_r^-\Lambda^1(\mathcal T_h)\times\mathcal P_r^-\Lambda^0(\mathcal T_h)\times K_h^\star$ such that
    \begin{equation*}
    \begin{aligned}
        A_h(\omega_h, p_h \mid \tau_h , q_h) +  B_h(\tau_h \mid   \lambda_h) = & 
        F(q_h)  & &\forall (\tau_h ,q_h ) \in D\mathcal P_r^-\Lambda^1(\mathcal T_h) \oplus   \mathrlap{ \mathcal P_r^-\Lambda^0(\mathcal T_h)}\\
        B_h(\omega_h \mid \mu_h ) =&  G_h(\mu_h) &&\forall \mu_h \in \mathrlap{ K_h^\star } 
    \end{aligned}
    \label{eq: altformulation}
    \end{equation*}
    is well-posed. 
\end{thm}
\begin{proof}
    The standard stability result for the discrete Hodge Laplacian \parencite[Theorem~7.3]{arnold2006finite}, specialized to $k=1$, gives an inf-sup constant $\alpha>0$ for $A_h$ that is independent of $h$. Indeed, the assumption $\harmonic\Lambda^1(M)=0$ from \Cref{sec: Mixed formulation}, together with the bounded commuting projections for the FEEC subcomplex, implies that the discrete harmonic space is trivial. The harmonic variable in the cited theorem therefore disappears, and a change of signs in the $0$-form variable and test function gives $A_h$.

    We now argue as in the proof of \Cref{thm: Well posed A + B}, using \Cref{lem: inf-sup finite element} for the inf-sup condition on $B_h$. The two inf-sup conditions give the claimed well-posedness.
\end{proof}

\subsection{Basic error estimate.}
We fix $\mathcal T$ as the coarsest triangulation and consider refinements $\mathcal T_h$. The right-hand side $F$ is unchanged under refinement, whereas the jump functional used by the finite element problem is obtained by interpolating $\gamma$. More precisely, let
\[
    \Pi_T^h \colon C^\infty\Lambda^k(T) \longrightarrow \mathcal P_r^-\Lambda^k(T)
\]
denote the canonical interpolator on a cell $T\in\mathcal T_h$. Applying it cellwise defines the broken interpolator
\[
    \Pi_{\mathcal T}^h \colon C^\infty\Lambda^k(\mathcal T)
    \longrightarrow D\mathcal P_r^-\Lambda^k(\mathcal T_h).
\]
If $\tau$ is globally conforming, its canonical degrees of freedom agree across facets, and hence $\Pi_{\mathcal T}^h\tau\in\mathcal P_r^-\Lambda^k(\mathcal T_h)$. In particular, this applies to the scalar interpolant $\Pi_{\mathcal T}^h p$ used below.
For sufficiently regular $\tau$, the standard interpolation estimate is
\[
    \|\Pi_{\mathcal T}^h\tau-\tau\|_{H\Lambda^k(\mathcal T)}
    \leq Ch^r\left(
    |\tau|_{H^r\Lambda^k(\mathcal T)}
    +|d^{\mathcal T}\tau|_{H^r\Lambda^{k+1}(\mathcal T)}
    \right).
\]
With this notation, let
\[
    G_h,G\colon\mathcal K^\star\Lambda^2(\mathcal T_h)\longrightarrow\mathbb R
\]
be given by
\[
    G_h(\mu)=\ljump \Pi_{\mathcal T}^h\gamma,\mu\rjump_{\mathcal T_h},
    \qquad
    G(\mu)=\ljump \gamma,\mu\rjump_{\mathcal T_h}.
\]
Let $(\omega^h,p^h,\lambda^h)$ be the solution of the refined infinite-dimensional problem whose jump functional is $G_h$. By \Cref{prop: saddle point T_h}, the refined problem with jump functional $G$ has the same primal variables $(\omega,p)$ as the original problem; we denote its multiplier by $\widetilde\lambda^h$. We assume that $\gamma$ is regular enough for the canonical interpolator and the estimates below.

\begin{lemma}\label{lem: stability projection}
    Let $(\omega^h,p^h,\lambda^h)$ and $(\omega,p,\widetilde\lambda^h)$ be the refined infinite-dimensional solutions described above. Then
    \[
        \|\omega^h-\omega\|_H+\|p^h-p\|_H
        \leq C\|\Pi_{\mathcal T}^h\gamma-\gamma\|_H,
    \]
    where $C$ is independent of $h$.
\end{lemma}
\begin{proof}
    Subtracting the two saddle-point problems shows that the differences solve
     \begin{equation*}
        \begin{aligned}
        A_h(\omega^h - \omega, p^h - p \mid \tau , q) +B_h(\tau \mid \lambda^h - \widetilde\lambda^h) = &
        0 &\forallspace &\forall (\tau, q) \in H\Lambda^1(\mathcal T_h)\mathrlap{\oplus H\Lambda^0(M)} \\
        B_h(\omega^h - \omega \mid \mu ) =& \ljump \Pi_{\mathcal T}^h \gamma - \gamma, \mu \rjump_{\mathcal T_h} &\forallspace& \forall \mu \in \mathcal K^\star\Lambda^2(\mathcal T_h)
        \end{aligned}
    \end{equation*}
    Moreover,
    \[
        \sup_{\substack{\mu\in\mathcal K^\star\Lambda^2(\mathcal T_h)\\
        \|\mu\|_{H^\star}=1}}
        \ljump \Pi_{\mathcal T}^h\gamma-\gamma,\mu\rjump_{\mathcal T_h}
        \leq \|\Pi_{\mathcal T}^h\gamma-\gamma\|_H.
    \]
    The stability estimate now follows by applying the abstract saddle-point theorem on $\mathcal T_h$ \parencite[Theorem~4.2.3]{boffi2013mixed}.
\end{proof}
The preceding stability result and standard mixed-method theory \parencite{boffi2013mixed} yield the following estimate.

\begin{thm}\label{thm: basic error estimate}
Let $(\omega^h,p^h,\lambda^h)\in H\Lambda^1(\mathcal T_h)\times H\Lambda^0(M)\times\mathcal K^\star\Lambda^2(\mathcal T_h)$ and
$(\omega_h,p_h,\lambda_h)\in D\mathcal P_r^-\Lambda^1(\mathcal T_h)\times\mathcal P_r^-\Lambda^0(\mathcal T_h)\times K_h^\star$ be the solutions of the infinite-dimensional and finite-dimensional saddle-point problems with the same jump functional $G_h$, respectively. Then
\begin{equation*}
\begin{aligned}
\| \omega^h - \omega_h \|_H + \| p^h - p_h \|_H \leq   &  C \left \{
\inf_{{\omega_I \in Z_h(\omega^h)}}\| \omega^h - \omega_I\|_H + \inf_{p_I\in \mathcal P_r^- \Lambda^0(\mathcal T_h)}\| p^h - p_I\| _H \right \} 
\end{aligned}
\end{equation*}
where
\[
    Z_h(\omega^h)=\left\{\omega_I\in D\mathcal P_r^-\Lambda^1(\mathcal T_h):
    \ljump\omega_I-\omega^h,\mu_h\rjump_{\mathcal T_h}=0
    \text{ for all }\mu_h\in K_h^\star\right\},
\]
and $C$ is independent of $h$.
\end{thm}
\begin{proof}
    \todo[inline,disable]{It does not matter much, but we might have to be explicit whether we use $\| (\omega, p)\| = \|\omega \| + \|p\|$ or $\sqrt {\|\omega\|^2 + \|p\|^2}$ to get the sharper constants in the end of the proof}
    We use a standard mixed-method error theorem \parencite[Theorem~5.2.4]{boffi2013mixed}. The requirements to use the theorem are in essence already proved. We use \Cref{thm: altformulation} for the inf-sup condition on $A_h$, \Cref{lem: inf-sup finite element} for the inf-sup condition on $B_h$, and \Cref{lem: Projection right} for the required compatibility of the discrete multiplier space. In terms of the constants $\alpha$ and $\|A_h\|$ that theorem gives
    \begin{equation}
        \big(\|\omega^h-\omega_h\|_H^2+\|p^h-p_h\|_H^2\big)^{1/2}
        \leq\frac{\|A_h\|}{\alpha}\Big\{
        \inf_{\omega_I\in Z_h(\omega^h)}\|\omega^h-\omega_I\|_H^2 + 
        \inf_{p_I\in\mathcal P_r^-\Lambda^0(\mathcal T_h)}\|p^h-p_I\|_H^2
        \Big\}^{1/2}
    \end{equation}
    And since $\| A_h \| \leq 1$, we may therefore take $C=\sqrt 2/\alpha$.\end{proof}
Analogous estimates are available for $\|\lambda^h - \lambda_h\|_{H^\star}$, but the multiplier is not of primary interest here.
\subsection{Constraints on the discrete connection form.}
\begin{lemma}\label{lem: interpolant constraint}
Assume that the solution $\omega$ has the regularity to apply $\Pi_{\mathcal T}^h$. Then $\Pi_{\mathcal T}^h\omega$ belongs to $Z_h(\omega^h)$. Moreover, for every facet $f=T_i\cap T_j\in\mathcal F_h$,
\[
    i_f^\ast\omega_h|_{T_j}-i_f^\ast\omega_h|_{T_i}
    =i_f^\ast\Pi_{\mathcal T}^h\omega|_{T_j}
    -i_f^\ast\Pi_{\mathcal T}^h\omega|_{T_i}.
\]
\end{lemma}
\begin{proof}
    Since $\ljump(\omega-\gamma)^\tangent\rjump=0$, we may write
    $\omega=\alpha+\gamma$ with $\alpha\in H\Lambda^1(M)$. The interpolant $\Pi_{\mathcal T}^h\alpha$ is conforming, so for every $\mu_h\in K_h^\star$,
    \[
        \ljump\Pi_{\mathcal T}^h\omega-\omega^h,\mu_h\rjump_{\mathcal T_h}
        =\ljump\Pi_{\mathcal T}^h(\alpha +\gamma)-\omega^h,\mu_h\rjump_{\mathcal T_h}
        =\ljump\Pi_{\mathcal T}^h\gamma-\omega^h,\mu_h\rjump_{\mathcal T_h}
        =0.
    \]
    Thus $\Pi_{\mathcal T}^h\omega\in Z_h(\omega^h)$. The same calculation with $\omega_h$ in place of $\omega^h$ gives
    \[
        \ljump\Pi_{\mathcal T}^h\omega-\omega_h,\mu_h\rjump_{\mathcal T_h}=0
        \qquad\forall\mu_h\in K_h^\star.
    \]
    The equality of the facet traces follows from \Cref{cor: discrete kernel} and \Cref{prop: jumps polynomial}.
\end{proof}

\subsection{Convergence rates.}
Recall that $(\omega, p, \lambda )$ satisfies \eqref{eq: variational problem}. Then, by considering only test forms with compact support in one simplex $T_i$, we have
\begin{align*}
    \langle \omega , dq \rangle_{L^2}  - \langle p,q \rangle_{L^2}&= 0, 
    & \forall q&\in C_0^\infty\Lambda^0(T_i)\\
    \langle d\omega , d\tau \rangle _{L^2} + \langle dp ,  \tau  \rangle_{L^2}&=0
    & \forall \tau &\in C_0^\infty \Lambda^1(T_i)
\end{align*}
The first equation implies that $\delta\omega_i=p_i\in H\Lambda^0(T_i)$, whereas the second implies that $\delta d\omega_i=-dp_i$. Hence
$d\delta\omega+\delta d\omega=\Delta\omega=0$ weakly in the interior of each $T_i$. Since the Hodge Laplacian is elliptic, interior elliptic regularity implies that $\omega_i$ is smooth, and consequently so is $p_i$. We assume enough piecewise regularity to apply $\Pi_{\mathcal T}^h$.

Combining the interpolation estimate above with standard estimates \parencite[p.~89]{arnold2018finite} gives
\begin{align*}
    \|\omega-\Pi_{\mathcal T}^h\omega\|_H
    &\leq Ch^r\left(
    |\omega|_{H^r\Lambda^1(\mathcal T)}
    +|d^{\mathcal T}\omega|_{H^r\Lambda^2(\mathcal T)}
    \right),\\
    \|p-\Pi_{\mathcal T}^h p\|_H
    &\leq Ch^r|p|_{H^{r+1}\Lambda^0(\mathcal T)}.
\end{align*}
By \Cref{lem: interpolant constraint}, $\Pi_{\mathcal T}^h\omega\in Z_h(\omega^h)$. Applying the triangle inequality, \Cref{thm: basic error estimate}, and then \Cref{lem: stability projection}, we obtain
\begin{align*}
    \|\omega-\omega_h\|_H+\|p-p_h\|_H
    \leq{}& \|\omega-\omega^h\|_H+\|p-p^h\|_H
    +\|\omega^h-\omega_h\|_H+\|p^h-p_h\|_H\\
    \leq{}& C\Bigl(
    \|\Pi_{\mathcal T}^h\gamma-\gamma\|_H
    +\|\omega-\Pi_{\mathcal T}^h\omega\|_H
    +\|p-\Pi_{\mathcal T}^h p\|_H
    \Bigr).
\end{align*}
This proves the following convergence estimate.
\begin{thm}\label{thm: Final convergence}
Let $(\omega,p,\lambda)$ and $(\omega_h,p_h,\lambda_h)$ be the continuous and discrete solutions described above, then
\begin{align*}
\|\omega-\omega_h\|_H+\|p-p_h\|_H
\leq{}& C\Bigl[
\|\Pi_{\mathcal T}^h\gamma-\gamma\|_H
+h^r\Bigl(
|\omega|_{H^r\Lambda^1(\mathcal T)}
+|d^{\mathcal T}\omega|_{H^r\Lambda^2(\mathcal T)}
\\[-0.2em]
&\hspace{11em}
+|p|_{H^{r+1}\Lambda^0(\mathcal T)}
\Bigr)\Bigr]\\
\leq{}& Ch^r\Bigl(
|\omega|_{H^r\Lambda^1(\mathcal T)}
+|d^{\mathcal T}\omega|_{H^r\Lambda^2(\mathcal T)}
\\[-0.2em]
&\hspace{6em}+|p|_{H^{r+1}\Lambda^0(\mathcal T)}
+|\gamma|_{H^r\Lambda^1(\mathcal T)}
+|d^{\mathcal T}\gamma|_{H^r\Lambda^2(\mathcal T)}
\Bigr),
\end{align*}
where the constant $C$ is independent of $h$, provided that $\omega, \gamma$ and $p$ are sufficiently regular. 
\end{thm}

The preceding estimate also controls the residual in the weak mixed formulation of the normal jump from \Cref{lem: Jump minimize with p}.
\begin{lemma}[Weak normal-jump estimate]
Let $(\omega,p,\lambda)$ and $(\omega_h,p_h,\lambda_h)$ denote the infinite- and finite-dimensional solution triples, respectively. With $F(q)=\ljump q,\phi^\normal\rjump$, we have
\begin{align}
&\sup_{\substack{q\in H\Lambda^0(M)\\ \|q\|_H=1}}
\left|
\langle\omega_h,dq\rangle_{L^2}
-\langle p_h,q\rangle_{L^2}-F(q)
\right| \notag\\
&\qquad={}
\sup_{\substack{q\in H\Lambda^0(M)\\ \|q\|_H=1}}
\left|
\langle\omega-\omega_h,dq\rangle_{L^2}
-\langle p-p_h,q\rangle_{L^2}
\right|
\leq \|\omega-\omega_h\|_{L^2}+\|p-p_h\|_{L^2}.
\end{align}
\end{lemma}
\begin{proof}
The equality follows from \eqref{eq: variational problem1}. By the Cauchy--Schwarz inequality, the expression inside the second supremum is bounded by
\[
\bigl(\|\omega-\omega_h\|_{L^2}^2+\|p-p_h\|_{L^2}^2\bigr)^{1/2}\|q\|_H,
\]
which proves the estimate.
\end{proof}
\section{Illustrative example and numerical test}
    \subsection{Example: Hopf fibration.}
    Let \begin{equation*}
        \begin{aligned}
            \mathbb S^3  = \{ (z_0, z_1) \in \mathbb C^2 : |z_0|^2 + |z_1|^2  = 1\} && \text{and} & &\mathbb S^2 = \{ (z, s) \in \mathbb C \times \mathbb R : |z|^2 + |s|^2  = 1 \}
        \end{aligned}
    \end{equation*}
    and consider the principal $\SSS^1$-bundle 
    \[ H \colon \SSS^3\to  \SSS^2
    \]
    which is given by 
    \[ 
    H\begin{pmatrix} z_0\\z_1\end{pmatrix} =\begin{pmatrix}2z_0^*z_1 \\|z_0|^2-|z_1|^2\end{pmatrix}
    \]
    restricted to the unit sphere $\SSS^3\subset \CC^2.$
    We use spherical coordinates $(\theta, \varphi)$ on $\SSS^2$, where $\theta$ is the polar angle and $\varphi$ is the azimuthal angle.
    
    We consider $\SSS^1= \{e^{i\xi} \colon 0\le \xi <2\pi\}$ as a Lie group with the standard product and Lie algebra
    \[\lie{s}_1 = T_1\SSS^1\simeq \RR,
    \]
    with exponential map 
    \begin{align*}
        \exp_{\SSS^1}\from \RR &\to \SSS^1\\
        \exp_{\SSS^1}(\xi)&= e^{i\xi}.
    \end{align*}
    Define $\sigma_N \from \SSS^2 \setminus\{(0,0,-1)\}\to \SSS^3$ by
    \[
        \sigma_N(\theta, \varphi) = 
        \begin{pmatrix}
        \cos\left(\frac{\theta}{2}\right)\\
        \sin\left(\frac{\theta}{2}\right)e^{i\varphi}
        \end{pmatrix}
    \]
    and 
    $\sigma_S\from \SSS^2 \setminus\{(0,0,1)\}\to \SSS^3$ by
    \[
        \sigma_S(\theta, \varphi) = 
        \begin{pmatrix}
        \cos\left(\frac{\theta}{2}\right)e^{-i\varphi}\\
        \sin\left(\frac{\theta}{2}\right)
        \end{pmatrix}
    \]
    Then $\sigma_{N}$ and $\sigma_S $ are smooth sections of the Hopf fibration on their respective domains. We see that \[\sigma_N(\theta, \varphi) =\sigma_S(\theta, \varphi)e^{i\varphi}\]
    Thus, the gauge transform is $g_{SN}(\theta, \varphi)=e^{i\varphi}$.

    Its trivialized derivative is
    \[ g_{SN}^{-1}\cdot dg_{SN} = i d\varphi
    \]
    Identifying the imaginary axis with the real axis, \eqref{eq: Ujumpcondition} becomes
    \begin{equation*}
        \omega_N - \omega_S =d\varphi
    \end{equation*}
    and the Yang--Mills variational problem can be formulated as 
    \[\begin{gathered}
    \text{minimize } \int_S \|d\omega_S\|^2\vol + \int_N \|d\omega_N\|^2\vol  \\
    \text{subject to }\omega_N - \omega_S =d\varphi \text{ on } N\cap S. 
    \end{gathered}
    \]
    where $N$ and $S$ are the north and south hemispheres, $N\cap S$ is the equator, and the constraint applies to both tangential and normal parts. In the language used in this paper, we have $\mathcal T = \{ N , S \} $.

    This problem can be solved analytically \parencite[see, e.g.,][Chap.~10.1]{Gockeler_Schucker_1987}, and the solution is
    \begin{equation*}\begin{aligned}
		\omega_N&= \frac{1}{2}\left(1-\cos \theta\right)d\varphi+df\\
		\omega_S&=-\frac{1}{2}\left(1+\cos \theta \right)d\varphi +df
		\end{aligned}
		\label{eq: gensol}
	\end{equation*}
	where $f\from \SSS^2\to \RR$ is an arbitrary smooth function representing the residual gauge freedom.
    
    If we fix the gauge using \eqref{eq: YMabelian min delta omega}, we get a unique solution
    \begin{equation*}	\begin{aligned}
			\omega_N&= \frac{1}{2}\left(1-\cos \theta\right)d\varphi\\
			\omega_S&=-\frac{1}{2}\left(1+\cos \theta \right)d\varphi
		\end{aligned}
		\label{eq: solution}
	\end{equation*}
    Instead of taking $N$ and $S$ to be the northern and southern hemispheres, respectively, we may use any subsets $N$ and $S$ such that $N \subset \mathbb S^2 \setminus \{(0,0,-1)\}$, $S \subset \mathbb S^2 \setminus \{ (0,0,1) \} $, $N \cup S = \mathbb S^2$, and $N\cap S$ is a 1-dimensional submanifold. In the numerical experiments, we also test $N\cap S = \{(x + iy,s)\in \mathbb S^2 \subset \mathbb C \times \mathbb R : x = s \}$. In this case, $d\varphi$ has both tangential and normal components on $N\cap S$.
    
\begin{figure}[!htbp]
\centering
\includegraphics[width=\linewidth]{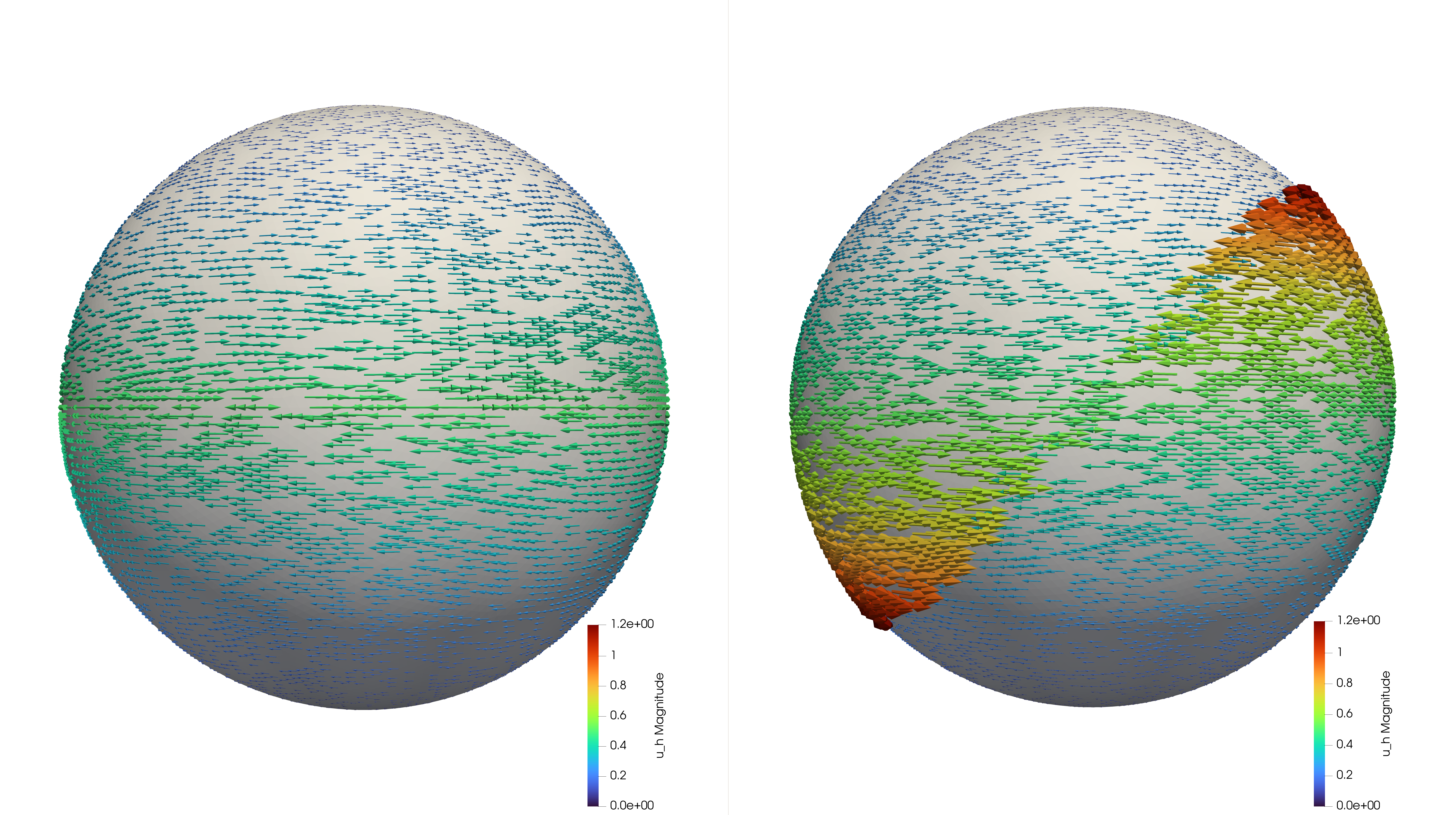} 
\caption{Visualization of $\omega_h$. Left: jump at $s=0$; right: jump at $s=x$.}
Alt text: Two spheres, one on the left side and one on the right side. On both sides, there is a vector field on the sphere, drawn with arrows. On the left side there is a jump discontinuity of the vector field at $s = 0$, and on the right side at $s = x$. 
\label{fig:hopf-vector}
\end{figure}

\begin{figure}[!htbp]
\centering
\includegraphics[width=\linewidth]{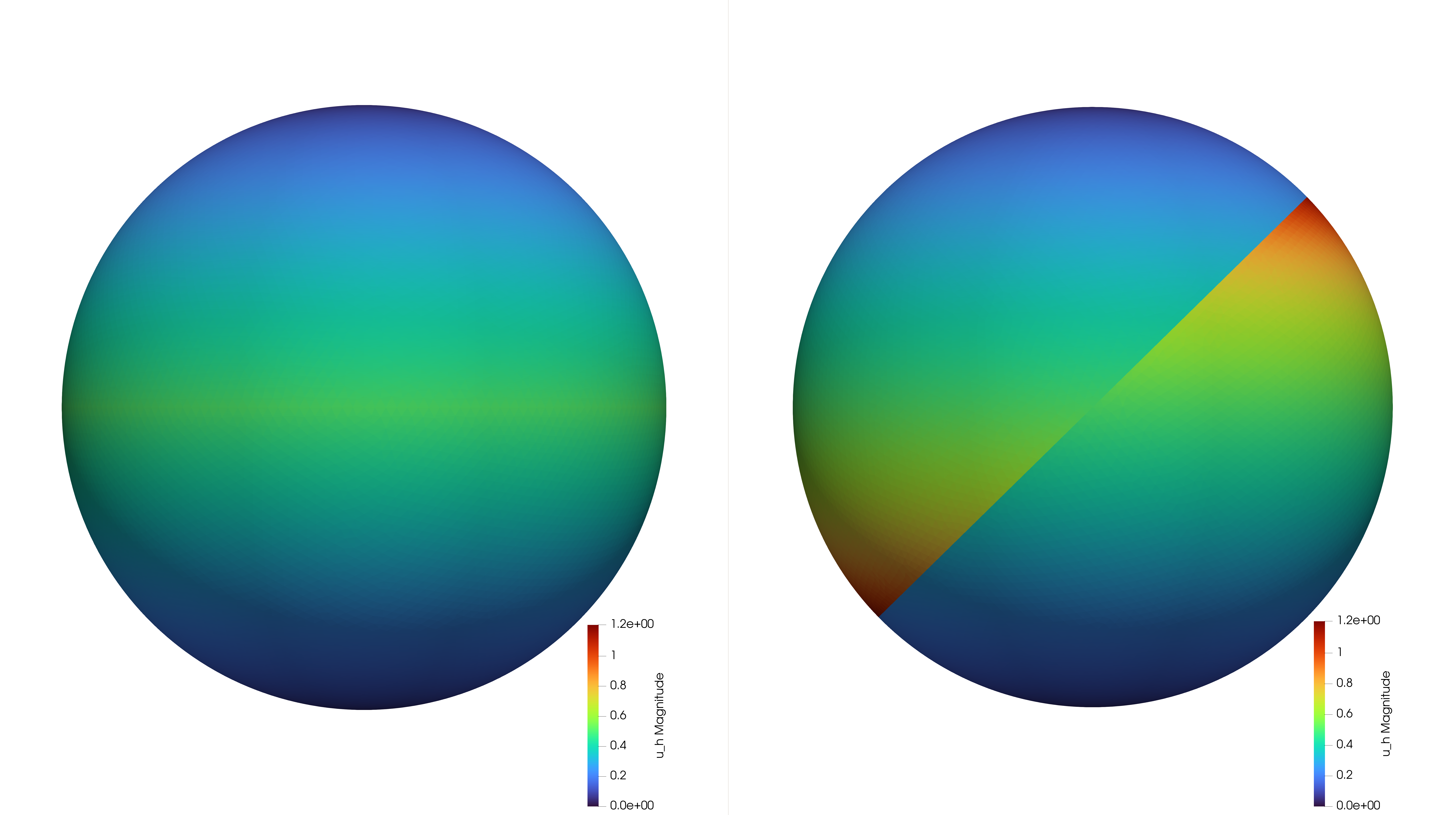} 
\caption{Magnitude of $\omega_h$. Left: jump at $s=0$; right: jump at $s=x$.}
Alt text: Two spheres, one on the left side and one on the right side. On both sides, the magnitude of the vector field from the previous figure is drawn with color. On the left side there is a jump discontinuity of the vector field at $s = 0$, and on the right side at $s = x$. 
\label{fig:hopf-magnitude}
\end{figure}
\subsection{Numerical experiment.}
We use the Yang--Mills connection on the Hopf bundle as our test problem. The implementation uses the Python library Firedrake \parencite{FiredrakeUserManual} and meshes generated with Gmsh \parencite{geuzaine2009gmsh}, which give a piecewise linear approximation of the sphere.

The spaces used are

\vspace{1ex}
\begin{tabular}{c|c}
\hline
$\mathcal P_r ^-\Lambda^0(\mathcal T_h)$  & Continuous Galerkin $(\mathtt{CG})$\\
$D\mathcal P_r ^-\Lambda^1(\mathcal T_h)$  & Broken Nédélec(first kind)  $(\mathtt{BrokenElement, N1curl})$\\
$V_h\simeq K_h^\star$ & Facet Nédélec (first kind) $(\mathtt{FacetElement, N1curl})$\\
\hline
\end{tabular}
\vspace{1ex}

Using a linear approximation of the sphere together with Nédélec spaces for 1-forms introduces a geometric variational crime \parencite{holst2012geometric}; we do not analyze the resulting error here.
The Facet Nédélec space for $K_h^\star$ requires an explanation, as the resulting spaces are the traces of
$D\mathcal P_r ^-\Lambda^1(\mathcal T_h)$, while $K_h^\star \subset H^\star \Lambda^2(\mathcal{T}_h)$.

By \Cref{prop: jumps polynomial}, we can write $\ljump \tau_h, \mu_h\rjump_{\mathcal T_h}= \langle Q_h\tau_h,R_h\mu_h\rangle_V$ for $\tau_h\in  D\mathcal P_r ^-\Lambda^1(\mathcal T_h)$, $\mu_h \in K_h^\star$, where $Q_h \tau_h$ and $R_h \mu_h$ are facet one-forms.

In the implementation, we use $R_h(K_h^\star)=V_h$ as the finite element space for the Lagrange multipliers and again denote the image multiplier by $\lambda_h$.
Finally, we assume that the jumps are specified by the vector proxy
$(g^{-1}dg)^\sharp$ on each facet.
That is, we solve for
$\omega_h\in D\mathcal P_r^-\Lambda^1(\mathcal T_h)$,
$p_h\in\mathcal P_r^-\Lambda^0(\mathcal T_h)$, and
$\lambda_h\in V_h$ such that
\begin{subequations}
\begin{align*}
\langle \omega_h,dq_h\rangle_{L^2}
-\langle p_h,q_h\rangle_{L^2}
&=\Phi_h^\normal(q_h)
&&\forall q_h\in\mathcal P_r^-\Lambda^0(\mathcal T_h),
\\
\langle d^{\mathcal T_h}\omega_h,d^{\mathcal T_h}\tau_h\rangle_{L^2}
+\langle dp_h,\tau_h\rangle_{L^2}
+\langle Q_h\tau_h,\lambda_h\rangle_{L^2(\mathcal F_h)}
&=0
&&\forall\tau_h\in D\mathcal P_r^-\Lambda^1(\mathcal T_h),
\\
\langle Q_h\omega_h,\mu_h\rangle_{L^2(\mathcal F_h)}
&=\Gamma_h^\tangent(\mu_h)
&&\forall\mu_h\in V_h.
\end{align*}
\end{subequations}
Here $Q_h\tau_h|_f=\ljump\tau_h\rjump_f$, and
$\langle\cdot,\cdot\rangle_{L^2(\mathcal F_h)}$ denotes the $L^2$-inner product
on the skeleton of the triangulation. Moreover,
\[
\begin{aligned}
\Phi_h^\normal(q_h)
&=
\left\langle
q_h|_{\mathcal F_h},\mathbf n\cdot(g^{-1}dg)^\sharp
\right\rangle_{L^2(\mathcal F_h)},\\
\Gamma_h^\tangent(\mu_h)
&=
\left\langle Q_h\Pi_{\mathcal T}^h\gamma,
\mu_h\right\rangle_{L^2(\mathcal F_h)}.
\end{aligned}
\]
In the implementation, we realize $\Gamma_h^\tangent$ directly from the vector proxy $(g^{-1}dg)^\sharp$ rather than constructing $\gamma$ explicitly. Since $\gamma$
has the smoothness required for $\Pi_{\mathcal T}^h$ and represents the prescribed tangential jump, the trace property of the canonical
interpolator shows that this gives the same discrete
tangential jump as $Q_h\Pi_{\mathcal T}^h\gamma$. For the Hopf example,
$g^{-1}dg=d\varphi$ after identifying the Lie algebra with $\mathbb R$,
so the prescribed jump field is known explicitly.
The numerical experiments exhibit linear convergence, including with higher-order polynomial spaces. The rate may be limited by the piecewise linear geometry, and is predicted by \parencite[Ex. 4.6]{holst2012geometric}. The function $p_h$ is zero up to numerical precision. The convergence data are summarized in \Cref{fig:hopf-convergence,tab:hopf-convergence-zero,tab:hopf-convergence-r}. In the convergence data, we have strictly speaking, interpolated the true solutions into $\mathcal P^-_{3}$-spaces(where we can handle discontinuity) before comparing with the numerical solution. 

\begin{figure}[!htbp]
\centering
\begin{tikzpicture}
\begin{axis}[
    width=0.82\linewidth,
    height=0.50\linewidth,
    x dir=reverse,
    xlabel={$\log_{10}(h)$},
    ylabel={$\log_{10}\left(\|\omega-\omega_h\|_{H}/\|\omega\|_{H}\right)$},
    grid=both,
    grid style={dashed, gray!50},
    legend pos=north west,
]
\addplot[
    blue,
    mark=*,
    mark size=2pt,
    line width=1.5pt,
] table[x=h, y=relative_error, col sep=comma] {error_convergence_data_w_0_radians.csv};
\addlegendentry{$s=0$}
\addplot[
    red,
    mark=x,
    mark size=2pt,
    line width=1.5pt,
] table[x=h, y=relative_error, col sep=comma] {error_convergence_data_w_0.7853981633974483_radians.csv};
\addlegendentry{$s=x$}
\end{axis}
\end{tikzpicture}
\caption{Relative $H$-norm error in $\omega_h$ for the equatorial interface $s=0$ and the rotated interface $s=x$.}
Alt text: Log-log plot of the relative error between the true solution, and the numerical solution. The x-axis shows the logarithm of h, and the y-axis the relative error. There are two graphs, one for the solution with s = 0, and one with s=x. The graphs are more or less straight, and parallel. 
\label{fig:hopf-convergence}
\end{figure}
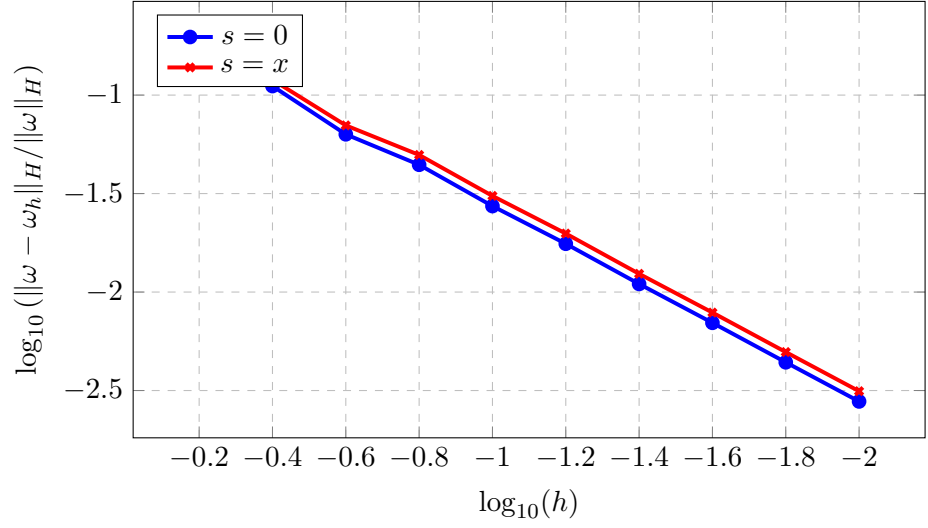

\pgfplotstableread[col sep=comma]{"error_convergence_data_w_0_radians.csv"}\omegaconvergencedata
\pgfplotstableread[col sep=comma]{"error_convergence_data_p_0_radians.csv"}\pconvergencedata

\pgfplotstablegetrowsof{\pconvergencedata}
\pgfmathtruncatemacro{\numprows}{\pgfplotsretval}

\pgfplotstablecreatecol[
    create col/assign/.code={
        \pgfmathtruncatemacro{\currentrow}{\pgfplotstablerow}
        \ifnum\currentrow<\numprows
            \pgfplotstablegetelem{\currentrow}{relative_error}\of{\pconvergencedata}
            \pgfkeyslet{/pgfplots/table/create col/next content}\pgfplotsretval
        \else
            \pgfkeyssetvalue{/pgfplots/table/create col/next content}{--}
        \fi
    }
]{p_error}\omegaconvergencedata

\begin{table}[htbp]
\centering
\pgfplotstabletypeset[
    columns={h,relative_error,p_error},
    columns/h/.style={
        column name={$ \log_{10}(h) $},
        fixed,
        fixed zerofill,
        precision=1
    },
    columns/relative_error/.style={
        column name={$ \log_{10}(\|\omega-\omega_h\|_H/\|\omega\|_H) $},
        fixed,
        fixed zerofill,
        precision=4
    },
    columns/p_error/.style={
        column name={$ \log_{10}(\|p-p_h\|_H) $},
        fixed,
        fixed zerofill,
        precision=4,
    },
    every head row/.style={
        before row=\hline,
        after row=\hline
    },
    every last row/.style={
        after row=\hline
    }
]\omegaconvergencedata
\caption{Errors for the computed connection form and the auxiliary variable $p$ under mesh refinement, with jump at $s=0$}
\label{tab:hopf-convergence-zero}
\end{table}

\pgfplotstableread[col sep=comma]{"error_convergence_data_w_0.7853981633974483_radians.csv"}\omegaconvergencedata
\pgfplotstableread[col sep=comma]{"error_convergence_data_p_0.7853981633974483_radians.csv"}\pconvergencedata

\pgfplotstablegetrowsof{\pconvergencedata}
\pgfmathtruncatemacro{\numprows}{\pgfplotsretval}

\pgfplotstablecreatecol[
    create col/assign/.code={
        \pgfmathtruncatemacro{\currentrow}{\pgfplotstablerow}
        \ifnum\currentrow<\numprows
            \pgfplotstablegetelem{\currentrow}{relative_error}\of{\pconvergencedata}
            \pgfkeyslet{/pgfplots/table/create col/next content}\pgfplotsretval
        \else
            \pgfkeyssetvalue{/pgfplots/table/create col/next content}{--}
        \fi
    }
]{p_error}\omegaconvergencedata

\begin{table}[htbp]
\centering
\pgfplotstabletypeset[
    columns={h,relative_error,p_error},
    columns/h/.style={
        column name={$ \log_{10}(h) $},
        fixed,
        fixed zerofill,
        precision=1
    },
    columns/relative_error/.style={
        column name={$ \log_{10}(\|\omega-\omega_h\|_H/\|\omega\|_H) $},
        fixed,
        fixed zerofill,
        precision=4
    },
    columns/p_error/.style={
        column name={$ \log_{10}(\|p-p_h\|_H) $},
        fixed,
        fixed zerofill,
        precision=4,
    },
    every head row/.style={
        before row=\hline,
        after row=\hline
    },
    every last row/.style={
        after row=\hline
    }
]\omegaconvergencedata
\caption{Errors for the computed connection form and the auxiliary variable $p$ under mesh refinement, with jump at $s=x$.}
\label{tab:hopf-convergence-r}
\end{table}

\clearpage
\section{Outlook and further work}
    A natural next step is to extend the methodology to the non-abelian case. This poses two challenges:

    \begin{enumerate}
    \item The Yang--Mills action functional 
    \[
    \mathcal{S}(\omega)= \frac12\sum_{T_i} \int_{T_i} \left\|d\omega_i + \frac12\left[\omega_i \wedge \omega_i\right]\right\|^2\vol
    \]
    becomes quartic
    
    \item The gauge conditions
    \[
    \sigma_\beta^\ast \omega = \Ad_ {g_{\alpha \beta }^{-1} } \sigma_\alpha ^\ast \omega +g_{\alpha \beta }^\ast \theta \quad \text{ on } U_\alpha  \cap U_\beta
    \]
    are not simple differences. The operator $\Ad_{g_{\alpha\beta}^{-1}}$ needs to be treated carefully.
    \end{enumerate}
    The variational problem becomes a nonlinear saddle-point problem. One possible approach is to impose each jump condition symmetrically from both sides and then solve the nonlinear saddle-point problem iteratively.

    A further generalization would be to consider a manifold equipped with a Lorentzian metric instead of a Riemannian one. In this case, even the linear problem becomes more challenging, since the bilinear form $\int_\Omega \langle d\omega, d\omega\rangle\vol$ is indefinite.

\printbibliography

\end{document}